\documentclass[10pt]{article}
\usepackage{amssymb,amsmath,amsthm}
\usepackage{amsfonts}
\usepackage{graphicx,subfigure,epstopdf}
\usepackage[usenames]{color}
\usepackage{url}
\usepackage{algorithm,algorithmic}
\usepackage{dsfont,verbatim}
\usepackage[margin=1.2in]{geometry}
\allowdisplaybreaks

\newtheorem{theorem}{Theorem}
\newtheorem{remark}{Remark}

\newtheorem{lemma}{Lemma}

\newtheorem{corollary}{Corollary}[section]

\numberwithin{equation}{section}

\newcommand{\R}{\mathbb{R}}

\def\d{{\rm d}}
\def\bPtau{\bar\partial_\tau}

\newcommand{\al}{\alpha}

\def\Dal{{\partial_t^\al}}

\def\d{{\rm d}}

\title{Subdiffusion with Time-Dependent Coefficients: Improved Regularity and Second-Order Time Stepping \thanks{The work of B. Jin is partially supported by UK EPSRC EP/T000864/1, that of B. Li is supported by Hong Kong RGC grant (Project No. 15300817), and that of  Z. Zhou by
 Hong Kong RGC grant (Project No. 25300818).}
}

\author{Bangti Jin\thanks{Department of Computer Science, University College London, Gower Street, London WC1E 6BT, UK
(\texttt{b.jin@ucl.ac.uk, bangti.jin@gmail.com})}
\and Buyang Li\thanks{Department of Applied Mathematics, The Hong Kong Polytechnic University, Hung Hom, Hong Kong.
(\texttt{buyang.li@polyu.edu.hk, libuyang@gmail.com})}
\and Zhi Zhou\thanks{ Department of Applied Mathematics, The Hong Kong Polytechnic University, Kowloon, Hong Kong.(\texttt{zhizhou@polyu.edu.hk, zhizhou0125@gmail.com})}
}
\date{\today}

\begin{document}

\maketitle

\begin{abstract}
This article concerns second-order time discretization of subdiffusion equations with time-dependent
diffusion coefficients. High-order differentiability and regularity estimates are established for subdiffusion
equations with time-dependent coefficients. Using these regularity results and a perturbation
argument of freezing the diffusion coefficient, we prove that the convolution quadrature generated by
the second-order backward differentiation formula, with proper correction at the first time step, can
achieve second-order convergence for both nonsmooth initial data and incompatible source term.
Numerical experiments are consistent with the theoretical results.\vskip5pt
\textbf{keywords}: subdiffusion, time-dependent coefficient, second-order backward differentiation formula, convolution quadrature, perturbation argument, error estimate
\end{abstract}


\section{Introduction}\label{sec:intro}

Let $\Omega\subset\mathbb{R}^d $ ($d=1,2,3$) be a convex polygonal domain with a boundary $\partial\Omega$. Consider
the following subdiffusion equation
\begin{align}\label{eqn:pde}
\left\{\begin{aligned}
&\Dal u(x,t) - \nabla\cdot(a(x,t)\nabla u(x,t)) = f(x,t), \quad &&(x,t)\in\Omega\times(0,T], \\
&u(x,t)=0,  &&(x,t)\in \partial\Omega\times(0,T], \\
&u(x,0)=u_0(x), &&x\in\Omega,
\end{aligned}
\right.
\end{align}
where $a(x,t):\Omega\times(0,T)\to \R^{d\times d}$ is a positive definite matrix-valued function, $f$ and $u_0 $ are the
source term and initial value, respectively, and
\begin{align}\label{eqn:RLderive}
   \Dal u(x,t)
   &: = \frac{1}{\Gamma(1-\alpha)}\int_0^t(t-s)^{-\alpha}\partial_s u (x,s)\d s ,
\end{align}
denotes the Caputo fractional time derivative of order
$\alpha\in(0,1)$ \cite[p. 70]{KilbasSrivastavaTrujillo:2006}.

In recent years, there has been a growing interest in the mathematical and numerical analysis of subdiffusion models due to
their diverse applications in describing subdiffusion processes arising from physics, engineering, biology
and finance. In a subdiffusion process, the mean squared particle displacement grows only sublinearly with time,
instead of growing linearly with time as in a normal diffusion process. At a microscopic level, such processes
can be adequately described by continuous time random walk, and accordingly, at a macroscopical level, the
probability density function of the particle appearing at certain time $t$ and location $x$ is described
by a subdiffusion model of the form \eqref{eqn:pde}. We refer interested readers to \cite{Metzler:2014,MetzlerKlafter:2000}
for a long list of applications arising in biology and physics. In the physical
literature, a time-dependent diffusion coefficient is often employed to study complex systems, e.g., turbulence
system \cite{HentschelProcaccia:1984,KlafterBlumenShlesinger:1987,FaLenzi:2005} and
cooling process in geology \cite{Dodson:1973,GarraGiustiMainardi:2018}; see also \cite{OrsingherPolito:2009,GarraOrsingherPolito:2015}
for its connection with birth-death processes.

The numerical analysis of the subdiffusion problem has been the topic of many recent investigations. In particular, a
large number of time-stepping schemes for approximating the Caputo derivative have been developed. The most
popular ones include convolution quadrature \cite{CuestaLubichPalencia:2006,JinLazarovZhou:SISC2016,JinLiZhou:2017sisc,BanjaiLopez:2019},
piecewise polynomial approximation \cite{SunWu:2006,LinXu:2007,Alikhanov:2015,LvXu:2016}, and discontinuous Galerkin method
\cite{McLeanMustapha:2009}. For a given smooth source term $f$ and initial value $u_0$, these schemes generally exhibit
only first-order convergence due to the inherent weak singularity of the solution at $t=0$. If the solution $u$ is smooth,
then higher-order convergence may be achieved, otherwise some modifications of the schemes
\cite{CuestaLubichPalencia:2006,JinLazarovZhou:SISC2016,JinLiZhou:2017sisc} or locally refined meshes 
\cite{McLeanMustapha:2009,StynesO'Riordan:2017,LiaoMcLean:2019} {(see also \cite{Brunner:2004} for related works 
in the context of Volterra integral equations)} can be used; see the recent survey \cite{JinLazarovZhou:2019} for 
further references. All these works focus on subdiffusion with a time-independent coefficient, i.e., $a(x,t)\equiv a(x)$.

When the diffusion coefficient $a(x,t)$ is time-dependent, the analysis of regularity of solutions and the development and convergence
analysis of numerical schemes are rather limited, despite its obvious practical importance. Many existing analytical techniques,
e.g., Laplace transform and separation of variables, are not directly applicable, due to the time-dependency of the coefficient $a(x,t)$.
Kubica and Yamamoto \cite{KubicaYamamoto:2017} proved the existence and uniqueness of a weak solution, and also
several regularity results. In this work, we present new regularity estimates in Theorems \ref{thm:reg-homo}
and \ref{thm:reg-inhomo}. For example, for $u_0\in L^2(\Omega)$ and
$f\equiv0$, under suitable conditions on $a(x,t)$, there holds
$\|\frac{\d^k}{\d t^k}(t^ku(t))\|_{L^{2}(\Omega)} \le c \| u_0  \|_{L^2(\Omega)}.$
Such an estimate provides one crucial tool for the error analysis of high-order time-stepping schemes.

 So far there are very few works on the numerical approximation of the model \eqref{eqn:pde} \cite{Mustapha:2017,JinLiZhou:2017variable}.
Mustapha \cite{Mustapha:2017} analyzed a spatially semidiscrete Galerkin finite element method (FEM) for the homogeneous problem, and showed
optimal order convergence by a novel energy argument. In essence, the approach extends the argument in \cite{LuskinRannacher:1982} for
standard parabolic problems to the fractional case. In the authors' prior work \cite{JinLiZhou:2017variable}, we developed a
different approach to analyze the spatially semidiscrete Galerkin scheme, as well as a fully discrete scheme based on convolution
quadrature (CQ) generated by backward Euler method (and L1 scheme), and showed optimal order convergence rates for both semidiscrete and
fully discrete schemes (up to a logarithmic factor), based on a perturbation argument and new regularity results. However, the discrete scheme
in \cite{JinLiZhou:2017variable} is only first order accurate in time. To the best of our knowledge, there is no proven second-
or higher-order accurate time-stepping scheme for the subdiffusion model with a time-dependent coefficient and nonsmooth problem data
in the literature. This contrasts sharply with the case of time-independent elliptic operators, for which there are several
strategies for devising high-order schemes, e.g., initial correction \cite{JinLiZhou:2017sisc}. These observations motivate
the present work.

In this article, we propose a second-order time-stepping scheme for problem \eqref{eqn:pde} with nonsmooth
initial data and incompatible source term. It is based on the CQ generated by the second-order
backward differentiation formula (BDF2), with suitable correction at the first step. The correction is inspired by the
recent works \cite{CuestaLubichPalencia:2006,JinLazarovZhou:SISC2016,JinLiZhou:2017sisc} and essential for restoring
the second-order convergence. Further, we present a complete error analysis in Section \ref{sec:error},
and prove a convergence rate $O(\tau^2)$ with $\tau$ being the time stepsize,
for any fixed $t_n>0$, of the scheme for both nonsmooth initial data and incompatible source term. The error analysis
relies heavily on new temporal regularity results for the model \eqref{eqn:pde} in Section \ref{sec:reg} and a refined
perturbation argument, which substantially extends the prior work \cite{JinLiZhou:2017variable}.
Specifically, the error analysis relies on suitable nonstandard bounds for problem data in the space $\dot H^{-\gamma}(\Omega)$ (cf.,
Lemma \ref{lem:sbd-0} and Theorem \ref{thm:err-inhomo}), and perturbation estimates at both $t=0$ and $t=t_m$ (cf., the
proof of Lemma \ref{lem:stability-homo}), which are substantially different from the one in \cite{JinLiZhou:2017variable}
which only requires estimates at $t=t_m$ for problem data in $L^2(\Omega)$. The new scheme, regularity results
and time discretization errors represent the main contributions of this work.

In the context of the standard parabolic counterpart with $L^2(\Omega)$-initial data and zero forcing term, Luskin
and Rannacher \cite{LuskinRannacher:1982} analyzed a fully discrete scheme based on Galerkin FEM in space
and the backward Euler method in time, and proved a first-order temporal convergence. Somewhat surprisingly,
Sammon \cite{Sammon:1983} proved that for standard parabolic problems with $L^2(\Omega)$
initial data, generally only second-order convergence can be achieved for a class of single step and linear multi-step
time stepping schemes (by ignoring the errors at starting steps).
The design and analysis of schemes with higher order accuracy remain largely elusive
for standard parabolic models with time-dependent elliptic operators and nonsmooth data. Thus, the development
and analysis of high-order time-stepping schemes for the model \eqref{eqn:pde} with general problem
data is still very challenging; see Section \ref{sec:scheme} for further discussions.

The rest of the paper is organized as follows. In Section \ref{sec:scheme}, we describe the proposed time-stepping scheme.
In Section \ref{sec:reg}, we prove new temporal regularity results, and in Section \ref{sec:error}, we give a
complete error analysis for both smooth and nonsmooth data. Finally in Section \ref{sec:numer}, we present numerical
results to complement the error analysis. Throughout, the notation $c$ denotes a generic constant which may differ at each occurrence, but it is always
independent of the time stepsize $\tau$, but may depend on the final time $T$.

\section{Derivation of the numerical scheme}\label{sec:scheme}

In this section, we construct a second-order time-stepping scheme for problem \eqref{eqn:pde} using
CQ generated by BDF2 with initial correction, derived from a perturbation argument. For notational
simplicity, we shall denote by $v(t)=v(\cdot,t)$ for a function $v$ defined on $\Omega\times(0,T]$.

Since the Riemann-Liouville derivative is equivalent to the Caputo one for functions
with zero initial value, we rewrite problem \eqref{eqn:pde} as
\begin{align}\label{eqn:reform-0}
  ^R\Dal (u(t)-u_0) + A(t) u(t) = f(t),
\end{align}
where the Riemann-Liouville derivative $^R\Dal \varphi(t)$ is defined by $^R\Dal \varphi(t)=
\frac{\d}{\d t}\frac{1}{\Gamma(1-\alpha)}\int_0^t(t-s)^{-\alpha}\varphi(s)\d s$, and the time-dependent elliptic
operator $A(t): H_0^1(\Omega)\cap H^2(\Omega)\rightarrow L^2(\Omega)$ is defined by
$$
A(t)\phi=- \nabla\cdot(a(x,t)\nabla \phi).
$$

Let $t_n=n\tau$, $n=0,1,\dots,N$, be a uniform partition of the interval $[0,T]$
with a time stepsize $\tau=T/N$. BDF2--CQ approximates the Riemann-Liouville
derivative $^{R}\partial_t^\alpha \varphi(t)$ at the time $t=t_n$ by
\begin{equation}\label{eqn:CQ}
\bar\partial_\tau^\alpha \varphi^n:= \frac{1}{\tau^\alpha}\sum_{j=0}^n b_j \varphi^{n-j}\quad\mbox{with $\varphi^n=\varphi(t_n)$},
\end{equation}
where the weights $\{b_j\}_{j=0}^\infty$ are the coefficients in the power series expansion
\begin{equation}\label{eqn:delta}
\delta_\tau(\zeta)^\alpha
=\frac{1}{\tau^\alpha}\sum_{j=0}^\infty b_j\zeta ^j \quad \mbox{with}\quad \delta_\tau(\zeta ):=
\frac{ \zeta^2-4\zeta+3}{2\tau}
\end{equation}
If the function $\varphi$ is smooth and has sufficiently many vanishing derivatives at $t=0$,
then BDF2--CQ is second-order accurate pointwise in time \cite{Lubich:1986}
\cite[Theorem 3.1]{Lubich:1988}.

By employing \eqref{eqn:CQ} to discretize the term $^R\Dal (u(t)-u_0)$ in
\eqref{eqn:reform-0}, we obtain a BDF2--CQ scheme for \eqref{eqn:pde}:
given $u^0=u_0$, find $u^n$ such that
\begin{equation}\label{eqn:fully}
\bPtau^\alpha (u-u_0)^n  +  A(t_n) u^n = f(t_n),\quad n=1,2\ldots,N .
\end{equation}
This scheme generally has only first-order accuracy, instead of second-order accuracy, due to the low
regularity of the solution $u(t)$ at $t=0$, unless restrictive compatibility conditions on the initial
data $u_0$ and $f$ are satisfied (which guarantee good solution regularity at $t=0$). This has been
observed for many different time-stepping schemes for subdiffusion with a time-independent diffusion coefficient
\cite{JinLiZhou:2017sisc,JinLazarovZhou:SISC2016,CuestaLubichPalencia:2006}. Hence, the vanilla BDF2--CQ scheme
\eqref{eqn:fully} has to be modified in order to achieve second-order convergence for general data.

In this work, we propose the following time-stepping scheme:
\begin{equation}\label{eqn:fully-correct}
\left\{\begin{aligned}
\bPtau^\alpha (u-u_0)^1  +  A(t_1) u^1 + \tfrac12 A(0) u_0 &= f(t_1) +\tfrac12 f(0),\\
\bPtau^\alpha (u-u_0)^n  +  A(t_n)u^n &= f(t_n),\qquad n=2,3,\ldots,N ,
\end{aligned}\right.
\end{equation}
which is obtained by first rewriting problem \eqref{eqn:pde} into
\begin{equation*}
^R\partial_t^\alpha (u-u_0) + A(0)u(t) = F(t)
  \quad\mbox{with}\quad
F(t)=f(t) + (A(0)-A(t))u(t) ,
\end{equation*}
and then following \cite{JinLazarovZhou:SISC2016,JinLiZhou:2017sisc} to modify the first step as
\begin{equation*}
  \bPtau^\alpha(u-u_0)^1 + A(0)u^1 + \tfrac12 A(0)u_0 = F(t_1) + \tfrac{1}{2}F(0).
\end{equation*}
Then substituting the expression of $F(t)$ and collecting terms yield the correction in \eqref{eqn:fully-correct}.
In \eqref{eqn:fully-correct}, the term $A(0)u_0$ should be interpreted in a distributional
sense for weak initial data, e.g., $u_0\in L^2(\Omega)$.

Note that $F'(0)$ is generally not defined in $L^2(\Omega)$. Hence, the existing correction methods in
\cite{JinLiZhou:2017sisc} for higher-order BDFs cannot be applied directly. It is still very challenging to
develop higher-order time discretization methods for problem \eqref{eqn:pde} with nonsmooth problem data.
This seems to be open even for the standard parabolic counterpart \cite{Sammon:1983}.

\section{Regularity of solutions}\label{sec:reg}

We assume that the diffusion coefficient $a(x,t):\Omega\times(0,T)\to \mathbb{R}^{d\times d}$ satisfies
that for some real number $\lambda\geq 1$, integer $K\geq 2$ and $i,j=1,\ldots,d$:
\begin{align}
&\lambda^{-1}|\xi|^2\le a(x,t)\xi\cdot\xi \le \lambda|\xi|^2,\quad
\forall\, \xi\in\mathbb{R}^d, \,\, \forall\, (x,t)\in \Omega\times(0,T],\label{Cond-1} \\
&|\tfrac{\partial}{\partial t} a_{ij}(x,t)|+|\nabla_x\tfrac{\partial^k}{\partial t^k}a_{ij}(x,t)| \le c,
\,\,\forall\, (x,t)\in \Omega\times(0,T],k=0,\ldots,K+1,\label{Cond-2}
\end{align}
where $\cdot$ and $|\cdot|$ denote the standard Euclidean inner product and norm, respectively.
Under these conditions, there holds $D(A(t))=H^1_0(\Omega)\cap H^2(\Omega)$ for all $t\in[0,T]$. By the complex
interpolation method \cite{Triebel:1978}, this implies
$$
D(A(t)^\gamma)=\dot H^{2\gamma}(\Omega) ,\quad\forall\, t\in[0,T], \,\,\,\forall\,\gamma\in[0,1],
$$
where $\dot H^{2\gamma}(\Omega)=(L^2(\Omega),H^1_0(\Omega)\cap H^2(\Omega))_{[\gamma]}$ denotes the complex interpolation space
between $L^2(\Omega)$ and $H^1_0(\Omega)\cap H^2(\Omega)$. Equivalently, it can be defined via spectral decomposition of
the operator $A(t)$ \cite[Chapter 3]{Thomee:2006}. Let $\{(\lambda_j,\varphi_j)\}_{j=1}^n$ be the eigenpairs of $A(t)$ with multiplicity
counted and $\{\varphi_j\}_{j=1}^\infty$ be an orthonormal basis in $L^2(\Omega)$. Then the space $\dot{H}^{\gamma}
(\Omega)$ can be defined as
\begin{equation*}
  \dot H^{\gamma}(\Omega) = \Big\{v\in L^2(\Omega): \sum_{j=1}^\infty \lambda_j^{\gamma}(v,\varphi_j)^2<\infty\Big\}.
\end{equation*}
In particular, $\dot H^{2}(\Omega)=H^1_0(\Omega)\cap H^2(\Omega)$, $\dot H^{1}(\Omega)=H^1_0(\Omega)$ and $\dot H^{0}(\Omega)=L^2(\Omega)$.
For $\gamma\in[0,2]$ we also denote by $\dot H^{-\gamma}(\Omega)$ the dual space of $\dot H^{\gamma}(\Omega)$. Then the norm
of $\dot H^{-\gamma}(\Omega)$ satisfies
\begin{equation*}
\|v\|_{\dot H^{-\gamma}(\Omega)}
=\|A(t)^{-\frac{\gamma}{2}}v\|_{L^2(\Omega)}\quad\forall\, v\in\dot H^{-\gamma}(\Omega),\,\,\forall\, t\in[0,T].
\end{equation*}

In this section, we prove the following regularity results.
\begin{theorem}[Homogeneous problem]\label{thm:reg-homo}
If $a(x,t)$ satisfies \eqref{Cond-1}-\eqref{Cond-2}, $u_0\in \dot H^{2\gamma}(\Omega)$ with
$\gamma\in [0,1]$ and $f\equiv0$, then for all $t\in(0,T]$ and $k=0,\ldots,K$, the solution
$u(t)$ to problem \eqref{eqn:pde} satisfies
\begin{equation*}
   \Big\|  \frac{\d^k}{\d t^k}(t^{k} u(t))\Big\|_{\dot H^{2\beta}(\Omega)} \le c t^{-(\beta-\gamma)\alpha} \| u_0  \|_{\dot H^{2\gamma}(\Omega)} ,
   \quad\forall\, \beta\in [\gamma,1].
\end{equation*}
\end{theorem}

\begin{theorem}[Inhomogeneous problem]\label{thm:reg-inhomo}
If $a(x,t)$ satisfies \eqref{Cond-1}-\eqref{Cond-2}, $u_0\equiv0$, then for all
$t\in(0,T]$ and $k=0,\ldots,K$, the solution $u(t)$ to problem \eqref{eqn:pde}
satisfies for any $\beta\in[0,1)$
\begin{align*}
\Big\|\frac{\d^k}{\d t^k}(t^ku(t))\Big\|_{\dot H^{2\beta}(\Omega)}
 \leq& c\sum_{j=0}^{k-1}t^{(1-\beta)\alpha+j}\|f^{(j)}(0)\|_{L^2(\Omega)} \\
  &+ c t^k\int_0^{t}(t-s)^{(1-\beta)\alpha-1}\|f^{(k)}(s)\|_{L^2(\Omega)}\d s,
\end{align*}
and similarly for $\beta=1$,
\begin{align*}
\Big\|\frac{\d^k}{\d t^k}(t^ku(t))\Big\|_{\dot H^{2\beta}(\Omega)} \leq
c\sum_{j=0}^k t^{j}\|f^{(j)}(0)\|_{L^2(\Omega)} + ct^k \int_0^t\|f^{(k+1)}(s)\|_{L^2(\Omega)}\d s.
\end{align*}
\end{theorem}

\begin{remark}
These regularity results are identical with that for subdiffusion with a time-independent elliptic
operator \cite[Theorems 2.1--2.2]{SakamotoYamamoto:2011}, \cite[Theorem 2.1]{JinLazarovZhou:2019}.
{All the constants in Theorems \ref{thm:reg-homo} and \ref{thm:reg-inhomo} may grow
with $k$ and blow up as $K\to \infty$, but stay bounded for any finite $K$.}
Further, these constants are uniformly bounded as $\alpha\to1^-$, similar to the prior 
estimates in \cite[Remark 2.1]{JinLiZhou:2017variable}.
\end{remark}

Theorem \ref{thm:reg-inhomo} implies the following estimate for smooth initial data.
\begin{corollary}\label{cor:reg-homo}
If $a(x,t)$ satisfies \eqref{Cond-1}-\eqref{Cond-2}, $u_0\in \dot H^2(\Omega)$ and
$f\equiv 0$, then for $w(t)=u(t)-u_0$, for all $t\in(0,T]$ and $k=0,\ldots,K$, there holds
\begin{align*}
\Big\|\frac{\d^k}{\d t^k}(t^kw(t))\Big\|_{\dot H^{2\beta}(\Omega)}  \le c t^{(1-\beta)\alpha} \|  u_0 \|_{\dot H^2(\Omega)},\quad  \forall\, \beta\in [0,1].
\end{align*}
\end{corollary}
\begin{proof}
The function $w(t)$ satisfies
$\partial_t^\alpha w(t) + A(t)w(t) = -A(t) u_0$ with $w(0) = 0.$
Then the assertion follows directly from Theorem \ref{thm:reg-inhomo}.\qed
\end{proof}

The rest of this section is devoted to the proof of Theorems \ref{thm:reg-homo} and \ref{thm:reg-inhomo}.

\subsection{Preliminaries}
First, we recall some preliminary results \cite{JinLiZhou:nonlinear} on the solution representation
and smoothing properties of solution operators for subdiffusion with a time-independent coefficient, i.e.,
\begin{equation}\label{PDE-independent}
\Dal u(t) + A_*u(t) = g(t),\quad \forall t\in (0,T],\quad \ \mbox{with } u(0)=u_0,
\end{equation}
where $A_*=A(t^*)$, for some fixed $t_*\in[0,T]$ independent of $t\in(0,T]$. By means of Laplace transform, the solution $u$ of
\eqref{PDE-independent} can be represented by (cf. \cite[Section 2]{JinLazarovZhou:2019} and \cite[Section 2]{JinLiZhou:nonlinear})
\begin{align}\label{eqn:Sol-expr-u-const}
u(t)= F_*(t)u_0 + \int_0^t E_*(t-s) g(s) \d s ,
\end{align}
where the operators $F_*(t)$ and $E_*(t)$ are respectively defined by
\begin{align}
&F_*(t):=\frac{1}{2\pi {\rm i}}\int_{\Gamma_{\theta,\delta }}e^{zt} z^{\alpha-1} (z^\alpha+A_*)^{-1}\, \d z , \label{eqn:EF1}\\
&E_*(t):=\frac{1}{2\pi {\rm i}}\int_{\Gamma_{\theta,\delta}}e^{zt}  (z^\alpha+A_*)^{-1}\, \d z , \label{eqn:EF2}
\end{align}
with the contour $\Gamma_{\theta,\delta}$ (oriented with an increasing imaginary part):
\begin{equation}\label{contour-Gamma}
  \Gamma_{\theta,\delta}=\left\{z\in \mathbb{C}: |z|=\delta, |\arg z|\le \theta\right\}\cup
  \{z\in \mathbb{C}: z=\rho e^{\pm\mathrm{i}\theta}, \rho\ge \delta\} .
\end{equation}
Throughout, we choose a fixed angle $\theta \in(\frac{\pi}{2},\pi)$ so that
$$
z^{\al} \in \Sigma_{\al\theta}
\quad\mbox{for}\quad
z\in\Sigma_{\theta}:=\{z\in\mathbb{C}\backslash\{0\}: |{\rm arg}(z)|\leq\theta\} .
$$
From the definitions \eqref{eqn:EF1} and \eqref{eqn:EF2}, we deduce
\begin{equation}\label{eqn:AE=F'}
A_*E_*(t) = (I- F_*(t))' ,
\end{equation}
which follows by straightforward computation
\begin{align*}
  (I-F_*(t))' & = - \frac{1}{2\pi {\rm i}}\int_{\Gamma_{\theta,\delta }}e^{zt} z^{\alpha} (z^\alpha+A_*)^{-1}\, \d z\\
    & = -\frac{1}{2\pi {\rm i}}\int_{\Gamma_{\theta,\delta }}e^{zt} (I- A_*(z^\alpha+A_*)^{-1})\, \d z
    = A_*E_*(t).
\end{align*}

The next lemma summarizes the smoothing properties of $F_*(t)$ and $E_*(t)$,
where $\|\cdot\|$ denotes the operator norm from $L^2(\Omega)$ to $L^2(\Omega)$.
\begin{lemma}\label{lem:smoothing}
For any integer $k=0,1,\ldots,$ the operators $F_*$ and $E_*$ defined in
 \eqref{eqn:EF1}-\eqref{eqn:EF2} satisfy for any $t\in (0,T]$
\begin{align*}
{\rm(i)}\quad & \mbox{$t^{-\alpha}\|A_*^{-1}(I-F_*(t))\|+t^{1-\alpha}\|A_*^{-1}F_*'(t)\|\le c$};\\
{\rm(ii)}\quad & \mbox{$t^{k+1-\alpha}\|E_*^{(k)}(t)\|+t^{k+1}\|A_*E_*^{(k)}(t)\|+t^{k+1+\alpha}\|A_*^2E_*^{(k)}(t)\|\le c $};\\
{\rm(iii)}\quad & \mbox{$t^k\|F_*^{(k)}(t)\|+t^{k+\alpha}\|A_*F_*^{(k)}(t)\| \leq c $}.
\end{align*}
\end{lemma}
\begin{proof}
The assertions for $k=0,1$ were already given in \cite[Lemma 2.2]{JinLiZhou:2017variable}. The proof
for $k>1$ is similar. For example, in part (i), by \eqref{eqn:AE=F'} and choosing $\delta=t^{-1}$ in
the contour $\Gamma_{\theta,\delta}$ and letting $\hat z=tz$:
\begin{align*}
\|A_*^{-1}F_*'(t)\| &=\|E_*(t)\| \le\frac{1}{2\pi}\int_{\Gamma_{\theta,\delta }}e^{\Re(z)t}\|(z^\alpha+A_*)^{-1}\|\, |\d z| \\
&\le ct^{\alpha-1} \frac{1}{2\pi}\int_{\Gamma_{\theta,1}}e^{\Re(\hat z)} |\hat z|^{-\alpha} |\d \hat z| \\
&\le ct^{\alpha-1}\frac{1}{2\pi}\int_{\Gamma_{\theta,1}}e^{\cos(\theta)|\hat z|} (1+|\hat z|^{-1}) |\d \hat z| \le c t^{\alpha-1},
\end{align*}
and in part (iii) with $k=0$, $\|F_*(t)\|$ can be bounded by
\begin{align*}
\|F_*(t)\|&\le\frac{1}{2\pi}\int_{\Gamma_{\theta,\delta }}e^{\Re(z)t} |z|^{\alpha-1} \|(z^\alpha+A_*)^{-1}\|\, |\d z| \\
&\le\frac{1}{2\pi}\int_{\Gamma_{\theta,\delta }}e^{\Re(z)t} |z|^{-1} |\d z|\le c.
\end{align*}
The proof of \eqref{eqn:AE=F'} gives $A_*E_*(t)=-\frac{1}{2\pi\mathrm{i}}\int_{\Gamma_{\theta,\delta}}e^{zt}z^\alpha(z^\alpha +A_*)^{-1}\d z$, and since $\|A_*(z_\alpha+A_*)^{-1}\|\leq c$, we deduce
\begin{equation*}
  \|A_*^2 E_*(t)\| \leq \frac{1}{2\pi}\int_{\Gamma_{\theta,\delta}}e^{\Re(z)t}|z|^\alpha |\d z| \leq ct^{-1-\alpha}.
\end{equation*}
All other estimates can be proved similarly and the details are omitted. Note that all the
constants $c$ remain bounded as $\alpha\rightarrow 1^-$. \qed
\end{proof}

The following perturbation estimate {\cite[Corollary 3.1]{JinLiZhou:2017variable}} will be used extensively. In particular, it implies that $\|A(s)^{-1}A(t)\|\leq c$ for any $s,t\in [0,T]$, and by interpolation also, $\|A(s)^{-\beta}A(t)^\beta\|\leq c$ for any $\beta\in(0,1)$.
\begin{lemma}\label{lem:conti-A}
Under conditions \eqref{Cond-1}--\eqref{Cond-2}, for any $\beta\in[0,1]$, there holds
\begin{align}
  \|(I-A(t)^{-1}A(s))v\|_{\dot H^{2\beta}(\Omega)}&\leq c|t-s|\|v\|_{\dot H^{2\beta}(\Omega)}, \quad \forall v\in {\dot H^{2\beta}(\Omega)}.
\end{align}
\end{lemma}

The following regularity results for problem \eqref{eqn:pde} were proved in \cite{JinLiZhou:2017variable} (also see
\cite{KubicaYamamoto:2017,DongKim:2019} for related results under different assumptions).
\begin{theorem}\label{thm:reg-0}
Under conditions \eqref{Cond-1}--\eqref{Cond-2}, the solution $u(t)$ of problem \eqref{eqn:pde} satisfies the following estimates:
\begin{itemize}
\item[(i)] If $u_0\in \dot H^{2\gamma}(\Omega)$, with some $\gamma\in[0,1]$, and $f=0$, then
\begin{align*}
 \|u(t)\|_{H^2(\Omega)} \le ct^{-(1-\gamma)\alpha} \|u_0\|_{\dot H^{2\gamma}(\Omega)}\,\,\mbox{and}\,\,
    \| u'(t)\|_{L^2(\Omega)} \le c t^{\gamma\alpha-1} \|u_0  \|_{\dot H^{2\gamma}(\Omega)}.
\end{align*}
\item[(ii)] If $u_0=0$, $f\in C([0,T]; L^2(\Omega))$ and $\int_0^t  (t-s)^{\alpha-1}\| f'(s)
\|_{L^2(\Omega)}\,\d s<\infty$, then
\begin{align*}
   \|u'(t)\|_{L^2(\Omega)} \le c t^{\alpha-1}  \| f(0) \|_{L^2(\Omega)} +c\int_0^t (t-s)^{\alpha-1} \| f'(s)\|_{L^2(\Omega)} \,\d s.
\end{align*}
\end{itemize}
\end{theorem}

Theorem \ref{thm:reg-0} is a special case of Theorems \ref{thm:reg-homo} and \ref{thm:reg-inhomo}
corresponding to $(k,\beta)=(0,1)$ and $(k,\beta)=(1,0)$, respectively. These results were used in \cite{JinLiZhou:2017variable} to prove
first-order convergence of backward Euler CQ. But they are insufficient to prove second-order convergence of the
corrected BDF2--CQ scheme \eqref{eqn:fully-correct}, which requires the regularity results in Theorems \ref{thm:reg-homo}
and \ref{thm:reg-inhomo} for $k= 2$. Below, we prove Theorems \ref{thm:reg-homo} and \ref{thm:reg-inhomo}
for a general nonnegative integer $k$.

\subsection{Proof of Theorems \ref{thm:reg-homo} and \ref{thm:reg-inhomo}}

The overall proof strategy is to employ a perturbation argument \cite{JinLiZhou:nonlinear,JinLiZhou:2017variable}
and then to properly resolve the singularity. Specifically,
for any fixed $t_*\in(0,T]$, we rewrite problem \eqref{eqn:pde} into
\begin{align}\label{eqn:reform}
\left\{\begin{aligned}\Dal u(t) + A_*u(t) &=  (A_*-A(t))u(t) + f(t) ,\quad \forall t\in(0,T],\\
        u(0)&=u_0.
\end{aligned}\right.
\end{align}
By \eqref{eqn:Sol-expr-u-const}, the solution $u(t)$ of \eqref{eqn:reform} is given by
\begin{align}\label{eqn:Sol-expr-u}
u(t)&= F_*(t)u_0+ \int_0^t E_*(t-s)(f(s) + (A_*-A(s))u(s))\d s.
\end{align}

The objective is to estimate the $k$th temporal derivative $u^{(k)}(t):=\frac{\d^k}{\d t^k}u(t)$ in $\dot
H^{2\beta}(\Omega)$ for $\beta\in[0,1]$ using \eqref{eqn:Sol-expr-u}. However, direct differentiation of
$u(t)$ in \eqref{eqn:Sol-expr-u} with respect to $t$ leads to strong singularity that precludes the use of
Gronwall's inequality in Lemma \ref{lem:gronwall}, in order to handle the perturbation term. To overcome
the difficulty, we instead estimate $\|(t^{k+1}u(t))^{(k)}\|_{\dot H^{2\beta}(\Omega)}$ using the
expansion of $t^{k+1}=[(t-s)+s]^{k+1}$ in the the following expression:

\begin{align}
  t^{k+1} u(t) = &  \, t^{k+1} F_*(t)u_0+ t^{k+1}\int_0^t E_*(t-s) f(s)\d s \label{eqn:split}\\
     & + \sum_{m=0}^{k+1} \left(\begin{aligned} & \,\,\,\,\, m \\ &\,k+1 \end{aligned}\right) \int_0^t(t-s)^mE_*(t-s)(A_*-A(s))s^{{k+1}-m}u(s)\d s,\nonumber
\end{align}
where $(\begin{subarray}{c}m\\ k+1\end{subarray})$ denotes binomial coefficients.
One crucial part in the proof is to bound $k$th-order derivatives of the summands in \eqref{eqn:split}.

Now we can give the proof of Theorem \ref{thm:reg-homo}.

\begin{proof}[of Theorem \ref{thm:reg-homo}]
When $k=0$, setting $f=0$ and $t=t_*$ in \eqref{eqn:Sol-expr-u} yields
\begin{equation*}
  A_*^\beta u(t_*) = A_*^\beta F_*(t_*)u_0 + \int_0^{t_*}A_*^\beta E_*(t_*-s)(A_*-A(s))u(s)\d s ,
\end{equation*}
where $\beta\in[\gamma,1]$. By Lemmas \ref{lem:smoothing} and \ref{lem:conti-A},
\begin{align*}
\|A_*^\beta& u(t_*)\|_{L^2(\Omega)}
 \leq \|A_*^{\beta-\gamma} F_*(t_*)A_*^\gamma u_0\|_{L^2(\Omega)} \\
&\qquad+\int_0^{t_*}\|A_*E_*(t_*-s)\|\|A_*^\beta(I-A_*^{-1}A(s))u(s)\|_{L^2(\Omega)}\d s\\
& \leq ct_*^{-(\beta-\gamma)\alpha} \|A_*^\gamma u_0\|_{L^2(\Omega)}
 + c\int_0^{t_*}(t_*-s)\|A_*E_*(t_*-s)\| \|A_*^\beta u(s)\|_{L^2(\Omega)} \d s\\
& \leq ct_*^{-(\beta-\gamma)\alpha} \|u_0\|_{\dot H^{2\gamma}(\Omega)}  + c\int_0^{t_*}\|A_*^\beta u(s)\|_{L^2(\Omega)}\d s.
\end{align*}
This and Gronwall's inequality in Lemma \ref{lem:gronwall} {with $\mu=(\beta-\gamma)\alpha$} yield
\begin{align*}
\|A_*^\beta u(t_*)\|_{L^2(\Omega)} \leq c(1-(\beta-\gamma)\alpha)^{-1}t_*^{-(\beta-\gamma)\alpha} \|u_0\|_{\dot H^{2\gamma}(\Omega)}   .
\end{align*}
In particular, we have $\|A_*^{\frac{\beta+\gamma}{2}} u(t_*)\|_{L^2(\Omega)}
 \leq ct_*^{-\frac{\beta-\gamma}{2}\alpha} \|u_0\|_{\dot H^{2\gamma}(\Omega)}$,
with $c$ being bounded as $\alpha\rightarrow 1^-$. This estimate
and Lemmas \ref{lem:smoothing}(ii) and \ref{lem:conti-A} then imply
\begin{align*}
&\|A_*^\beta u(t_*)\|_{L^2(\Omega)} \le \|A_*^{\beta-\gamma} F_*(t_*)A_*^\gamma u_0\|_{L^2(\Omega)} \\
&\qquad+\int_0^{t_*}\|A_*^{\frac{\beta-\gamma}{2}}A_*E_*(t_*-s)\|\|A_*^{\frac{\beta+\gamma}{2}}(I-A_*^{-1}A(s))u(s)\|_{L^2(\Omega)}\d s\\
& \leq ct_*^{-(\beta-\gamma)\alpha} \|A_*^\gamma u_0\|_{L^2(\Omega)}
 + c\int_0^{t_*}(t_*-s)\|A_*^{\frac{\beta-\gamma}{2}}A_*E_*(t_*-s)\| \|A_*^{\frac{\beta+\gamma}{2}}u(s)\|_{L^2(\Omega)} \d s \\
& \leq c\Big(t_*^{-(\beta-\gamma)\alpha}
 + \int_0^{t_*}(t_*-s)^{-\frac{\beta-\gamma}{2}\alpha}  s^{-\frac{\beta-\gamma}{2}\alpha} \d s \Big)\|u_0\|_{\dot H^{2\gamma}(\Omega)}
\le ct_*^{-(\beta-\gamma)\alpha} \|u_0\|_{\dot H^{2\gamma}(\Omega)}.
\end{align*}
Equivalently, we have
\begin{align*}
\|A_*^\beta t_*u(t_*)\|_{L^2(\Omega)}
&\le ct_*^{1-(\beta-\gamma)\alpha} \|u_0\|_{\dot H^{2\gamma}(\Omega)},
\end{align*}
where $c$ is bounded as $\alpha\to1^-$. This proves the assertion for $ k=0$.

Next we prove the case $1\leq k\le K$ using mathematical induction. Suppose that the assertion holds up to $k-1< K$, and we prove it for $k\leq K$.
Indeed, by Lemma \ref{lem:homo} below,
\begin{align*}
    &\Big\|A_*^\beta\frac{\d^k}{\d t^k}\int_0^t(t-s)^{m} E_*(t-s) (A_*-A(s))s^{k+1-m}u(s) \d s|_{t=t_*}\Big\|_{L^2(\Omega)}\nonumber\\
    \leq &c t_*^{-(\beta-\gamma)\alpha+1}\|u_0\|_{\dot H^{2\gamma}(\Omega)} + c \int_0^{t_*}\|A_*^\beta(s^{k+1}u(s))^{(k)}\|_{L^2(\Omega)}\d s,
\end{align*}
where $m=0,1,\ldots, k+1$. Meanwhile, the estimates in Lemma \ref{lem:smoothing} imply
\begin{align*}
  \big\|A_*^\beta \big(t^{k+1}F_*(t)u_0\big)^{(k)}\big\|_{L^2(\Omega)}\leq ct^{-(\beta-\gamma)\alpha+1}\|u_0\|_{\dot H^{2\gamma}(\Omega)} .
\end{align*}
By applying $A_*^\beta \frac{\d^k}{\d t^k}$ to \eqref{eqn:split}
and using the last two estimates, we obtain
\begin{align*}
  \big\|A_*^\beta(t^{k+1}u(t))^{(k)}|_{t=t_*}\big\|_{L^2(\Omega)}
  \leq& ct_*^{-(\beta-\gamma)\alpha+1} \|u_0\|_{\dot H^{2\gamma}(\Omega)}\\
   &+ c\int_0^{t_*}\|A_*^\beta(s^{k+1}u(s))^{(k)}\|_{L^2(\Omega)}\d s.
\end{align*}
Last, applying the standard Gronwall's inequality, we complete the induction step and also the proof of the theorem.
\qed
\end{proof}

In the proof of Theorem \ref{thm:reg-homo}, we have used the following result.
\begin{lemma}\label{lem:homo}
Under the conditions of Theorem \ref{thm:reg-homo}, for $m=0,\ldots,k+1$, there holds
\begin{align*}
    &\Big\|A_*^\beta\frac{\d^k}{\d t^k}\int_0^t(t-s)^{m}E_*(t-s) (A_*-A(s))s^{k+1-m}u(s) \d s|_{t=t_*}\Big\|_{L^2(\Omega)}\\
    \leq &c t_*^{-(\beta-\gamma)\alpha+1}\|u_0\|_{\dot H^{2\gamma}(\Omega)} + c \int_0^{t_*}\Big\|A_*^\beta(s^{k+1}u(s))^{(k)}\Big\|_{L^2(\Omega)}\d s.
\end{align*}
\end{lemma}
\begin{proof}
Denote the integral on the left hand side by ${\rm I}_m(t)$, and let $v_m=t^mu(t)$ and $W_m(t)=t^{m}E_*(t)$. Direct
computation using product rule and changing variables gives that for any $0\le m\le k$, there holds
\begin{align*}
  &{\rm I}_m^{(k)}(t)
     =  \frac{\d^{k-m}}{\d t^{k-m}}\int_0^t W_{m}^{(m)}(t-s)(A_*-A(s))v_{k-m+1}(s) \d s\\
     = & \frac{\d^{k-m}}{\d t^{k-m}}\int_0^t W_{m}^{(m)}(s)(A_*-A(t-s))v_{k-m+1}(t-s) \d s\\
     = &  \int_0^tW_{m}^{(m)}(s) \frac{\d^{k-m}}{\d t^{k-m}}\big((A_*-A(t-s))v_{k-m+1}(t-s)\big)\d s\\
     = & \sum_{\ell=0}^{k-m} \bigg(\begin{array}{c} \ell\\ k-m\end{array}\bigg)\underbrace{\int_0^tW_{m}^{(m)}
     (s)(A_*-A(t-s))^{(k-m-\ell)}v_{k-m+1}^{(\ell)}(t-s)\d s}_{{\rm I}_{m,\ell}(t)}.
\end{align*}
Next we bound the integrand
$$
\widetilde {\rm I}_{m,\ell}(s):=W_m^{(m)}(A_*-A(t_*-s))^{(k-m-\ell)}v_{k-m+1}^{(\ell)}(t_*-s)
$$
of the integral ${\rm I}_{m,\ell}(t_*)  .$
We shall distinguish between $\beta\in[\gamma,1)$ and $\beta=1$.
First we analyze the case  $\beta\in [\gamma,1)$. When $\ell<k$, by Lemmas \ref{lem:smoothing}(ii) and \ref{lem:conti-A}
and the induction hypothesis, we bound the integrand $\widetilde{\rm I}_{m,\ell}(s)$ by
\begin{align*}
& \|A_*^\beta \widetilde {\rm I}_{m,\ell}(s)\|_{L^2(\Omega)} \\
&\leq \|A_*^\beta W_{m}^{(m)}(s)\|\|(A_*-A(t_*-s))^{(k-m-\ell)} v_{k-m+1}^{(\ell)}(t_*-s)\|_{L^2(\Omega)}\\
  & \leq \left\{\begin{aligned}
     cs^{(1-\beta)\alpha-1}s\|A_*v_{k-m+1}^{(k-m)}(t_*-s)\|_{L^2(\Omega)}, &\quad \ell=k-m,\\
     cs^{(1-\beta)\alpha-1}\|A_*v_{k-m+1}^{(\ell)}(t_*-s)\Big\|_{L^2(\Omega)}, &\quad \ell<k-m,
    \end{aligned}\right.\\
& \leq \left\{\begin{aligned}
    & cs^{(1-\beta)\alpha}(t_*-s)^{1-(1-\gamma)\alpha}\|A_*^\gamma u_0\|_{L^2(\Omega)}, & \ell=k-m,\\
    & cs^{(1-\beta)\alpha-1}(t_*-s)^{k-m-\ell+1 -(1-\gamma)\alpha}\|A_*^\gamma u_0\|_{L^2(\Omega)}, & \ell<k-m.
    \end{aligned}\right.
\end{align*}
Similarly for the case $\ell=k$ (and thus $m=0$), there holds
\begin{align*}
\|A_*^\beta \widetilde {\rm I}_{0,k}(s)\|_{L^2(\Omega)}
 & \leq \|A_*E_*(s)\|\|A_*^\beta(I-A_*^{-1}A(t_*-s))v_{k+1}^{(k)}\|_{L^2(\Omega)}\\
   & \leq c\|A_*^\beta v_{k+1}^{(k)}(t_*-s)\|_{L^2(\Omega)}.
\end{align*}
Thus, for $0\le m\le k$ and $\ell=k-m$, upon integrating from $0$ to $t_*$, we obtain
\begin{equation*}
\|A_*^\beta{\rm I}_m^{(k)}(t_*)\|_{L^2(\Omega)}
\le ct_*^{2+(\gamma-\beta)\alpha}\|A_*^\gamma u_0\|_{L^2(\Omega)}
+c\int_0^{t_*} \|A_*^\beta v_{k+1}^{(k)}(s)\|_{L^2(\Omega)} \d s,
\end{equation*}
and similarly for $0\le m\le k$ and $\ell<k-m$,
\begin{align*}
\|A_*^\beta{\rm I}_m^{(k)}(t_*)\|_{L^2(\Omega)}
\le& c((1-\beta)\alpha)^{-1}t_*^{1+(\gamma-\beta)\alpha}\|A_*^\gamma u_0\|_{L^2(\Omega)}\\
&+c\int_0^{t_*} \|A_*^\beta v_{k+1}^{(k)}(s)\|_{L^2(\Omega)} \d s.
\end{align*}
Meanwhile, for $m=k+1$, we have
\begin{equation*}
A_*^\beta{\rm I}_{k+1}^{(k)}(t_*) = \int_0^{t_*} A_*^{\beta+1-\gamma}W_{k+1}^{(k)}(t_*-s) A_*^\gamma (I-A_*^{-1}A(s))u(s) \d s,
\end{equation*}
and consequently, by Lemmas \ref{lem:smoothing}(ii) and \ref{lem:conti-A} and the induction hypothesis,
\begin{align*}
&\|A_*^\beta{\rm I}_{k+1}^{(k)}(t_*)\|_{L^2(\Omega)}\\
&\le \int_0^{t_*} \|A_*^{\beta+1-\gamma}W_{k+1}^{(k)}(t_*-s)\| \|A_*^\gamma(I-A_*^{-1}A(s))u(s)\|_{L^2(\Omega)} \d s\\
&\le c\int_0^{t_*} (t_*-s)^{1-(\beta-\gamma)\alpha} \|A_*^\gamma u(s)\|_{L^2(\Omega)} \d s
\le c t_*^{2+(\gamma-\beta)\alpha}\|A_*^\gamma u_0\|_{L^2(\Omega)} .
\end{align*}

In the case $0\le m\le k$ and $\ell<k-m$, the preceding estimates require $\beta\in[0,1)$.
When $0\le m\le k$, $\ell<k-m$ and $\beta=1$, we apply the identity \eqref{eqn:AE=F'} and
rewrite $A_*{\rm I}_{m,\ell}(t_*)$ as
\begin{align*}
   A_*{\rm I}_{m,\ell}(t_*) = \int_0^{t_*} (s^m(I-F_*(s))')^{(m)}(A_*-A(t_*-s))^{(k-m-\ell)} v_{k-m+1}^{(\ell)}(t_*-s)\d s.
\end{align*}
Then integration by parts and product rule yield
\begin{align}
A_*{\rm I}_{m,\ell}(t_*)
& = -\int_0^{t_*} D(s)(A_*-A(t_*-s))^{(k-m-\ell+1)}v_{k-m+1}^{(\ell)}(t_*-s)\d s \nonumber\\
   &\quad -\int_0^{t_*} D(s)(A_*-A(t_*-s))^{(k-m-\ell)}v_{k-m+1}^{(\ell+1)}(t_*-s)\d s \nonumber\\
&\quad -D(0)(A_*-A(t_*-s))^{(k-m-\ell)}|_{s=0}v_{k-m+1}^{(\ell)}(t_*) ,\label{eqn:int-part}
\end{align}
with
\begin{equation*}
  D(s) = \left\{\begin{aligned}
    I-F_*(s), &\quad m=0,\\
    (s^{m}(I-F_*(s))')^{(m-1)}, &\quad m>0.
  \end{aligned}\right.
\end{equation*}
By Lemma \ref{lem:smoothing}(iii), $\|D(s)\|\leq c$, and thus the preceding argument with Lemmas \ref{lem:smoothing}
and \ref{lem:conti-A} and the induction hypothesis allows bounding the integrand $A_*\widetilde{I}_{m,\ell}(s)$ of
\eqref{eqn:int-part} by
\begin{align*}
\|A_*\widetilde{\rm I}_{m,\ell}(s) \|_{L^2(\Omega)} & \leq c(t_*-s)^{k-\ell-(1-\gamma)\alpha} \|u_0\|_{\dot H^{2\gamma}(\Omega)} \\
   &\quad + \left\{\begin{aligned}
        c\|A_*v_{k+1}^{(k)}(t_*-s)\|_{L^2(\Omega)}, &\ \ \ell = k-1,\\
        c(t_*-s)^{k-1-\ell-(1-\gamma)\alpha}\|u_0\|_{\dot H^{2\gamma}(\Omega)}, &\ \ \ell<k-1,
\end{aligned}\right.
\end{align*}
where for $\ell=k-1$, we have $m=0$ and hence $D(0)=0$.

Combining the last estimates and then integrating from $0$ to $t_*$ in $s$, we obtain the desired assertion of Lemma \ref{lem:homo}. All the estimates are based on Lemmas \ref{lem:smoothing} and \ref{lem:conti-A}, and thus the constants $c$ in Lemma \ref{lem:homo} is bounded as $\alpha\rightarrow 1^-$.
\qed
\end{proof}

The proof of Theorem \ref{thm:reg-inhomo} is similar to that of Theorem \ref{thm:reg-homo}.
The lengthy and technical proof is deferred to Appendix \ref{app:inhomo}.

\section{Error analysis}\label{sec:error}

In this section, we present error estimates for the scheme \eqref{eqn:fully-correct}. To
this end, let $w(t)=u(t) - u(0)$, which satisfies the equation
\begin{align}\label{eqn:pde-fix}
\left\{
\begin{aligned}
\Dal w(t) + A(0)w(t) &= g(t),\quad \forall t>0,\\
w(0)&=0.
\end{aligned}
\right.
\end{align}
with
\begin{equation*}
   g(t):=(A(0)-A(t))w(t) -  A(t) u_0 + f(t).
\end{equation*}
Then the error $e^n:=u^n-u(t_n)$ of the numerical solution $u^n$ is given by
\begin{align}
e^n=w^n-w(t_n), \quad\mbox{with\,\,\,} w^n=u^n-u_0.
\end{align}
We also introduce an intermediate solution $\overline w^n$ defined by
\begin{equation}\label{eqn:fully-correct-01}
\left\{\begin{aligned}
\bPtau^\alpha  \overline w ^1  +  A(0)\overline w^1   &= g(t_1)+\tfrac{1}{2}g(t_0),\\
\bPtau^\alpha  \overline w ^n  +  A(0)\overline w^n &= g(t_n), \quad n=2,3,\ldots, N.
\end{aligned}\right.
\end{equation}
which is the numerical approximation of \eqref{eqn:pde-fix} with the source
$g(t)$. Using $\overline w^n$, we further decompose the error $e^n$ into
\begin{equation*}
e^n = (w^n - \overline w^n) + (\overline w^n - w(t_n))=:\varrho^n+\vartheta^n ,
\end{equation*}
where $\vartheta^n$ is the error due to time discretization of problem  \eqref{eqn:pde-fix} with
a ``time-independent'' operator $A(0)$, and $\varrho^n$ is the error between two numerical solutions due to
the perturbation of the source term.

It suffices to estimate the two terms $\varrho^n$ and $\vartheta^n$.
The analysis for $\vartheta^n$ will employ the following nonstandard error estimates.

\begin{lemma}\label{lem:sbd-0}
Let $u(t)$ be the solution of problem \eqref{PDE-independent} with $u_0\equiv0$ and $u^n$, with $u^0=0$, defined by
\begin{align*}
\left\{\begin{aligned}
\bPtau^\alpha u^1 + A(t_*)u^1  &= g(t_1) + \tfrac{1}{2}g(0),\\
\bPtau^\alpha u^n + A(t_*)u^n  &= g(t_n), \qquad\qquad n = 2,\ldots,N.
\end{aligned}\right.
\end{align*}
Then the following statements hold.
\begin{itemize}
\item[(i)] If $\beta,\gamma\in[0,1)$ and $\beta+\gamma<1$, then
\begin{equation*}
  \begin{aligned}
  \|u(t_n) - u^n\|_{\dot H^{2\beta}(\Omega)} & \leq  c\tau^2\Big(t_n^{(1-\beta)\alpha-2}\| g(0)\|_{L^2(\Omega)}
  + t_n^{(1-\beta)\alpha-1}\| g'(0)\|_{L^2(\Omega)} \\
    &\quad + \int_0^{t_n} (t_{n+1}-s)^{(1-\beta-\gamma)\alpha-1} \|g''(s) \|_{\dot H^{-2\gamma}(\Omega)} \,\d s\Big).
  \end{aligned}
\end{equation*}
\item[(ii)] If $\beta=1$, then
\begin{equation*}
\begin{split}
\|u(t_n)-u^n\|_{\dot H^2(\Omega)}
\le &   c \tau^2 \Big(t_n^{-2} \| g(0) \|_{L^2(\Omega)} +  t_n^{-1}\|g'\|_{C([0,\tau];L^2(\Omega))}\\
  &+\int_\tau^{t_n}(t_{n+1}-s)^{-1}\|g''(s)\|_{L^2(\Omega)}\d s \Big) .
\end{split}
\end{equation*}
\end{itemize}
\end{lemma}

Lemma \ref{lem:sbd-0} can be proved using discrete Laplace transform (generating function technique)
similarly as the error estimation for CQ--BDF$k$ \cite{JinLiZhou:2017sisc}.
This type of error estimation yields an error bound directly from a contour integral, while the constant
produced from a contour integral is bounded as $\alpha\rightarrow 1^-$.
We will use Lemma \ref{lem:sbd-0} and a perturbation argument to bound
$\vartheta^n$ and $\varrho^n$, respectively, and derive error estimates for numerical solutions.

For the convenience of error analysis, we further split $w(t)$ into $w(t)=w_0(t)+w_1(t)$, where $w_0(t)$ and $w_1(t)$
are respectively solutions of
\begin{align}
 \Dal w_0(t) + A(0) w_0(t) &= (A(0) -A(t))w(t), && \mbox{with }w_0(0)=0,\label{eqn:w0}\\
 \Dal w_1(t) + A(0) w_1(t) &= -  A(t) u_0+f(t), && \mbox{with }w_1(0)=0.\label{eqn:w1}
\end{align}
Correspondingly, we split $\overline{w}^n$ into $\overline{w}^n=\overline w_0^n+\overline w_1^n$, defined by $\overline{w}_0^0=0$,
\begin{equation}\label{def-w0n}
\bPtau^\alpha \overline w_0^n  + A(0)\overline w_0^n = (A(0)- A(t_n))w(t_n)  ,\quad n=1,2,3,\ldots,N, \\
\end{equation}
and $\overline{w}_1^0=0$ and
\begin{equation}\label{def-w1n}
\left\{\begin{aligned}
&\bPtau^\alpha \overline w_1^1  +  A(0)\overline w_1^1  =  -A(t_1) u_0 - \tfrac{1}{2}A(0)u_0 + f(t_1)+\tfrac{1}{2}f(t_0),\\
&\bPtau^\alpha  \overline w_1^n  +  A(0)\overline w_1^n = -A(t_n) u_0 + f(t_n),\quad n=2,3,\ldots,N,\\
\end{aligned}\right.
\end{equation}
The functions $\overline{w}_0^n$ and $\overline{w}_1^n$  approximate $w_0(t_n)$ and $w_1(t_n)$, respectively.

\subsection{Error analysis for the homogeneous problem}

Now we analyze the scheme \eqref{eqn:fully-correct} for the homogeneous problem with $f\equiv 0$.
First, we bound the function $g(t)=(A(0)-A(t))w(t)$ in equation \eqref{eqn:w0}.
\begin{lemma}\label{lem:g}
Let Assumptions \eqref{Cond-1}--\eqref{Cond-2} hold. For the function $g(t)=(A(0)-A(t))w(t)$, the following statements hold when $f\equiv0$.
\begin{itemize}
  \item[(i)] $u_0\in \dot H^{2}(\Omega)$ and $\beta\in[0,1]$, then $\|g'(0)\|_{L^2(\Omega)}+t^{1-\alpha\beta}\|g''(t)\|_{\dot H^{-2\beta}(\Omega)}\leq c\|u_0\|_{\dot H^2(\Omega)}$.
  \item[(ii)] $u_0\in L^2(\Omega)$, then $\|g'(t)\|_{\dot H^{-2}(\Omega)} + t\|g''(t)\|_{\dot H^{-2}(\Omega)} \le c \| u_0 \|_{L^2(\Omega)}$.
\end{itemize}
\end{lemma}
\begin{proof}
By Theorem \ref{thm:reg-homo} and triangle inequality,
$
\|w(t)\|_{\dot H^2(\Omega)}\leq \|u(t)\|_{\dot H^2(\Omega)}+
\|u_0\|_{\dot H^2(\Omega)}\leq c\|u_0\|_{\dot H^2(\Omega)} .
$
Thus, by Lemma \ref{lem:conti-A},
\begin{align*}
  \|g'(t)\|_{L^2(\Omega)} & \leq \|(A(0)-A(t))w'(t)\|_{L^2(\Omega)} + \|A'(t)w(t)\|_{L^2(\Omega)}\\
     & \leq ct\|u'(t)\|_{\dot H^2(\Omega)} + c\|w(t)\|_{\dot H^2(\Omega)}\leq c\|u_0\|_{\dot H^2(\Omega)},
\end{align*}
Thus, $\|g'(0)\|_{L^2(\Omega)} \leq c\|u_0\|_{\dot H^2(\Omega)}$. Since
$g''(t) = (A(0)-A(t))w''(t) - 2A'(t)w'(t) - A''(t)w(t),$
it follows from Corollary \ref{cor:reg-homo} and Theorem \ref{thm:reg-homo} that for $\beta\in[0,1]$
\begin{align*}
\|g''(t)\|_{\dot H^{-2\beta}(\Omega)}&=
\|(A(0)-A(t))w''(t) - 2A'(t)w'(t) - A''(t)w(t)\|_{\dot H^{-2\beta}(\Omega)} \\
&\leq ct \|w''(t)\|_{\dot H^{2-2\beta}(\Omega)} + c\|w'(t)\|_{\dot H^{2-2\beta}(\Omega)}+c\|w(t)\|_{\dot H^{2-2\beta}(\Omega)}\\
& \leq ct^{\alpha\beta-1} \| u_0 \|_{\dot H^2(\Omega)}.
\end{align*}
Similarly, when $u_0\in L^2(\Omega)$, repeating the preceding argument shows (ii).
\end{proof}

The next lemma bounds $\vartheta^n=\overline w^n-w(t_n)$.
\begin{lemma}\label{lem:theta}
Let conditions \eqref{Cond-1}-\eqref{Cond-2} hold, and $w$ be the solution to problem \eqref{eqn:pde-fix}
with $f\equiv0$. Let $\vartheta^n:=\overline w^n-w(t_n)$. Then there hold
\begin{align*}
&\| \vartheta^n\|_{\dot H^{2\beta}(\Omega)}
\le c \tau^2 t_n^{\alpha(1-\beta)-2} \| u_0 \|_{\dot H^2(\Omega)}, &&\forall \beta\in[0,1/2),\\
&\|\vartheta^n\|_{L^2(\Omega)}
\le c \tau^2 t_n^{-2}\ell_n\|u_0\|_{L^2(\Omega)}, &&\mbox{with }\ell_n = \log(1+t_n/\tau).
\end{align*}
\end{lemma}
\begin{proof}
Using the decompositions $w(t)=w_0(t)+w_1(t)$ and $\overline{w}^n=\overline w_0^n+\overline w_1^n$ defined
in \eqref{eqn:w0}-\eqref{eqn:w1} and \eqref{def-w0n}-\eqref{def-w1n}, respectively, we have
\begin{equation}\label{eqn:split-vartheta}
\|\vartheta^n\|_{\dot H^{2\beta}(\Omega)} \leq \| \overline{w}_0^n-w_0(t_n) \|_{\dot H^{2\beta}(\Omega)} + \|\overline{w}_1^n-w_1(t_n) \|_{\dot H^{2\beta}(\Omega)}.
\end{equation}
We discuss the cases $u_0\in \dot H^{2}(\Omega)$ and $u_0\in L^2(\Omega)$, separately.\medskip

\noindent{\it Case (i): $u_0\in \dot H^{2}(\Omega)$.} Lemma \ref{lem:sbd-0}(i) with $g(t)=A(t)u_0$, for $\beta\in[0,1/2)$, implies
\begin{align}\label{eqn:est-w1}
\| \overline{w}_1^n - w_1(t_n) \|_{\dot H^{2\beta}(\Omega)}
&\leq c \tau^2 t_n^{(1-\beta)\alpha-2} \| u_0 \|_{\dot H^{2}(\Omega)}.
\end{align}
For $g(t) = (A(0) - A(t))w(t)$ and any $\beta\in[0,1/2)$, Lemmas \ref{lem:sbd-0}(i) and \ref{lem:g} imply
\begin{align*}
& \|\overline{w}_0^n-w_0(t_n) \|_{\dot H^{2\beta}(\Omega)}\\
&\le
c\tau^2t_n^{(1-\beta)\alpha-1}\| g'(0)\|_{L^2(\Omega)} + c\tau^2\int_0^{t_n} (t_{n+1}-s)^{(1-2\beta)\alpha-1} \| g''(s) \|_{\dot H^{-2\beta}(\Omega)} \,\d s\\
&\le
c\tau^2\Big(t_n^{(1-\beta)\alpha-1} + \int_0^{t_n}(t_{n+1}-s)^{(1-2\beta)\alpha-1}s^{\alpha\beta-1}\d s\Big)\| u_0\|_{\dot H^{2}(\Omega)} \\
&\le
c \tau^2 t_n^{\alpha(1-\beta)-1} \| u_0 \|_{\dot H^{2}(\Omega)}.
\end{align*}
This and \eqref{eqn:est-w1} yield the desired estimate for $u_0\in \dot H^{2}(\Omega)$.\medskip

\noindent{\it Case (ii): $u_0\in L^2(\Omega)$.} By Lemma \ref{lem:sbd-0}(ii), we have
\begin{equation*}
   \|\overline{w}_1^n- w_1(t_n)\|_{L^2(\Omega)} \le c \tau^2 t_n^{ -2}\ell_n \| u_0 \|_{L^2(\Omega)}.
\end{equation*}
Meanwhile, by Lemmas \ref{lem:sbd-0}(ii) and \ref{lem:g}, we have
\begin{align*}
&\|\overline{w}_0^n - w_0(t_n)\|_{L^2(\Omega)}\\
&\le c\tau^2\Big(t_n^{-1}\|g'\|_{C([0,\tau];\dot H^{-2}(\Omega))} + \int_\tau^{t_n}(t_{n+1}-s)^{-1}\| g''(s)\|_{\dot H^{-2}(\Omega)}\,ds\Big)\\
&\le c\tau^2\Big( t_n^{-1} + \int_\tau^{t_n} (t_{n+1}-s)^{-1}s^{-1} \d s\Big) \|u_0\|_{L^2(\Omega)} \le c\tau^2 t_n^{-1}\ell_n\|  u_0 \|_{L^2(\Omega)}.
\end{align*}
These two estimates give the second assertion, completing the proof.\qed
\end{proof}

We need a temporally semidiscrete solution operator $ E_{\tau,m}^{n}$ defined by
\begin{align}\label{eqn:semi-operator}
E_{\tau,m}^n = \frac{1}{2\pi\mathrm{i}}\int_{\Gamma_{\theta,\delta}^\tau } e^{zn\tau} ({ \delta_\tau(e^{-z\tau})^\alpha}+A(t_m))^{-1}\,\d z ,
\end{align}
with the contour $\Gamma_{\theta,\delta}^\tau$ given by
\begin{equation}\label{eqn:contour-tau}
   \Gamma_{\theta,\delta}^\tau :=\{ z\in \Gamma_{\theta,\delta}:|\Im(z)|\le {\pi}/{\tau} \},
\end{equation}
oriented with an increasing imaginary part. The following smoothing property of the
operator $E_{\tau,m}^n$ holds \cite[Lemma 4.3]{JinLiZhou:2017variable}: for any $\beta\in[0,1]$
\begin{equation}\label{eqn:E-smoothing}
  \|A(t_m)^\beta E_{\tau,m}^n\|\leq c(t_n+\tau)^{(1-\beta)\alpha-1},\quad n=0,1,\dots,N.
\end{equation}

We have the following $L^2(\Omega)$ stability for $\varrho^n$.
\begin{lemma}\label{lem:stability-homo}
Let conditions \eqref{Cond-1}-\eqref{Cond-2} be fulfilled, and $u$ the solution to problem \eqref{eqn:pde}
with $f\equiv0$. Let $\varrho^n= w^n-\overline w^n$. Then with $\ell_n=\log(1+t_n/\tau)$, there holds
\begin{equation*}
   \| \varrho^m  \|_{L^2(\Omega)} \le c \tau\sum_{k=1}^m\|\varrho^k\|_{L^2(\Omega)}+\left\{\begin{array}{ll}
     c\tau^2 t_m^{\alpha-1}\|u_0 \|_{\dot H^2(\Omega)}, & \mbox{if } u_0\in \dot H^{2}(\Omega),\\
     c\tau^2 t_m^{-1}\ell_m^2\|u_0\|_{L^2(\Omega)}, & \mbox{if } u_0\in L^2(\Omega).
     \end{array}\right.
\end{equation*}
\end{lemma}
\begin{proof}\,
It follows from \eqref{eqn:fully-correct} and \eqref{eqn:fully-correct-01} that $\varrho^n$
satisfies $\varrho^0=0$ and
\begin{align*}
\bPtau^\alpha  \varrho ^n  +  &A(t_m) \varrho^n = \bPtau^\alpha (w^n-\overline{w}^n) + A(t_m)(w^n-\overline{w}^n)\\
   & = (A(t_m)-A(t_n))w^n - (A(t_m)-A(0))\overline{w}^n-(A(0)-A(t_n))w(t_n)\\
   & = (A(t_m) - A(t_n)) \varrho^n - (A(t_n)-A(0))\vartheta^n  ,\quad n=1,2,\ldots,N.
\end{align*}
Using the operator $E_{\tau,m}^n$ in \eqref{eqn:semi-operator}, $\varrho^m$ is represented by
\begin{equation*}
  \varrho^m = \tau \sum_{k=1}^m E_{\tau,m}^{m-k} \big[(A(t_m) - A(t_k)) \varrho^k - (A(t_k)-A(0))\vartheta^k\big].
\end{equation*}
Consequently, by triangle inequality,
\begin{align*}
 \|\varrho^m  \|_{L^2(\Omega)}&\le \tau\sum_{k=1}^m \|E_{\tau,m}^{m-k}(A(t_m)-A(t_k))\varrho^k\|_{L^2(\Omega)}\\
   &\quad +\tau \sum_{k=1}^m\|E_{\tau,m}^{m-k}(A(t_k)-A(0))\vartheta^k\|_{L^2(\Omega)}:={\rm I} + {\rm II}.
\end{align*}
For the term $\rm I$, by \eqref{eqn:E-smoothing} with $\beta=1$ and Lemma \ref{lem:conti-A}, we have
\begin{align*}
 \|A(t_m)E_{\tau,m}^{m-k}\|\|(I-A(t_m)^{-1}A(t_k))\varrho^k\|_{L^2(\Omega)}\
   & \leq ct_{m-k+1}^{-1}t_{m-k}\|\varrho^k\|_{L^2(\Omega)},
\end{align*}
and thus
\begin{align}
 {\rm I}& \leq c\tau \sum_{k=1}^m \|\varrho^k\|_{L^2(\Omega)}. \label{eqn:bdd-rhoI}
\end{align}
For the term ${\rm II}$, we discuss the cases $u_0\in \dot H^{2}(\Omega)$ and $u_0\in L^2(\Omega)$ separately.\\
{\it Case (i): $u_0\in \dot H^{2}(\Omega)$}. The estimate \eqref{eqn:E-smoothing} with $\beta=\frac34$,
Lemmas \ref{lem:conti-A} and \ref{lem:theta} with $\beta=\frac{1}{4}$ imply that ${\rm II}_{m,k}=\|E_{\tau,m}^{m-k}(A(t_k)-A(0))\vartheta^k\|_{L^2(\Omega)}$ is bounded by
\begin{align*}
  {\rm II}_{m,k}
\le ct_{m-k+1}^{\frac{\alpha}{4}-1} t_k \| \vartheta^k\|_{\dot H^{\frac12}(\Omega)}
\le c\tau^2t_{m-k+1}^{\frac{\alpha}{4}-1} t_k^{\frac{3\alpha}{4}-1}\|u_0\|_{\dot H^2(\Omega)}
\end{align*}
and further, since $\tau\sum_{k=1}^m t_{m-k+1}^{\frac{\alpha}{4}-1} t_k^{\frac{3\alpha}{4}-1}
\leq ct_m^{\alpha-1}$, there holds
\begin{equation*}
  {\rm II}\leq \tau\sum_{k=1}^m {\rm II}_{m,k} \leq c\tau^2 t_m^{\alpha-1}\|u_0 \|_{\dot H^2(\Omega)}.
\end{equation*}
{\it Case (ii): $u_0\in L^2(\Omega)$.} By \eqref{eqn:E-smoothing} and Lemmas \ref{lem:theta} and \ref{lem:conti-A},
\begin{align*}
  {\rm II}_{m,k} &\leq \|E_{\tau,m}^{m-k}A(t_m)\|\|A(t_m)^{-1}A(0)\|
  \|(I-A(0)^{-1}A(t_k))\vartheta^k\|_{L^2(\Omega)}\\
 &\le ct_{m-k+1}^{-1} t_k   \| \vartheta^k\|_{L^2(\Omega)}\le c\tau^2\ell_m t_{m-k+1}^{-1} t_k^{-1}\|u_0\|_{L^2(\Omega)}.
\end{align*}
This and the inequality $\tau \sum_{k=1}^mt_{m-k+1}^{-1} t_k^{-1} \leq ct_m^{-1}\ell_m$ yield
\begin{align*}
  {\rm II}  
\le c\tau^2 t_m^{-1}\ell_m^2 \|u_0 \|_{L^2(\Omega)}.
\end{align*}
In either case, combining the bounds on ${\rm I}$ and ${\rm II}$ gives the desired assertion.
\qed
\end{proof}

Now we can derive error estimates for the homogeneous problem.
\begin{theorem}\label{thm:err-smooth}
Let $u$ and $u^n$ be the solutions to problems \eqref{eqn:pde} and \eqref{eqn:fully-correct} with $f\equiv0$, respectively. Then with $\ell_n=\log(1+t_n/\tau)$,
there holds
\begin{equation*}
  \| u(t_n)-u^n \|_{L^2(\Omega)}   \le \left\{\begin{array}{ll}
    c \tau^2 t_n^{\alpha-2}\|u_0 \|_{\dot H^2(\Omega)}, & \mbox{if } u_0\in \dot H^{2}(\Omega),\\
    c \tau^2 t_n^{-2}\ell_n^2\|u_0 \|_{L^2(\Omega)}, & \mbox{if } u_0\in L^2(\Omega).
  \end{array}\right.
\end{equation*}
\end{theorem}
\begin{proof}
It follows directly from Lemma \ref{lem:stability-homo} that
\begin{equation*}
   \| \varrho^m  \|_{L^2(\Omega)} \le c \tau\sum_{k=1}^m\|\varrho^k\|_{L^2(\Omega)}+\left\{\begin{array}{ll}
     c\tau^2 t_m^{\alpha-1}\| u_0 \|_{\dot H^2(\Omega)}, & \mbox{if } u_0\in \dot H^{2}(\Omega),\\
     c\tau^2 t_m^{-1}\ell_m^2\|u_0\|_{L^2(\Omega)}, & \mbox{if } u_0\in L^2(\Omega).
     \end{array}\right.
\end{equation*}
Thus, by the discrete Gronwall's inequality from Lemma \ref{lem:gronwall-discrete} {(with $\mu=1-\alpha$)} and Lemma
\ref{lem:gronwall-discrete-log},
\begin{equation*}
 \| \varrho^m  \|_{L^2(\Omega)}  \le \left\{\begin{array}{ll}
   c \tau^2 t_m^{\alpha-1}\| u_0 \|_{\dot H^2(\Omega)}, & \mbox{if } u_0\in \dot H^{2}(\Omega),\\
   c \tau^2 t_m^{-1}\ell_m^2 \|u_0 \|_{L^2(\Omega)}, & \mbox{if }u_0\in L^2(\Omega).
   \end{array}\right.
\end{equation*}
This, Lemma \ref{lem:theta} and the triangle inequality complete the proof.
The preceding estimates are based on Lemma \ref{lem:stability-homo} and Lemmas
\ref{lem:gronwall-discrete}--\ref{lem:gronwall-discrete-log}. In particular,
applying Lemma \ref{lem:gronwall-discrete} to the case $u_0\in \dot H^{2}(\Omega)$
yields a constant $c$ depending on $1/\alpha$. Therefore, the constants $c$ in
Theorem \ref{thm:err-smooth} is bounded as $\alpha\rightarrow 1^-$.\qed
\end{proof}

\begin{remark}
The error estimate for $u_0\in \dot{H}^2(\Omega)$ in Theorem \ref{thm:err-smooth} is identical with that for
the case of a time-independent elliptic operator, and that for nonsmooth initial data is
also nearly identical, up to the factor $\ell_n^2$ \cite{JinLazarovZhou:SISC2016}. The $\ell_n$
factor is also present for backward Euler
convolution quadrature \cite{JinLiZhou:2017variable} for subdiffusion, and backward Euler method \cite{LuskinRannacher:1982} and
general single-step and multi-step methods \cite{Sammon:1983} for standard parabolic problems with a time-dependent coefficient.
\end{remark}

\subsection{Error analysis for the inhomogeneous problem}
Now we analyze the scheme \eqref{eqn:fully-correct} for $u_0\equiv 0$.
We need the following inequality.
\begin{lemma}\label{lem:bdd-kernel}
For any $\beta\in(0,1/2)$ and $s\in [0,t_m]$, the following inequality holds
\begin{equation*}
  \tau\sum_{k=1}^mt_{m-k+1}^{\beta\alpha-1}(t_{k+1}-s)^{(1-2\beta)\alpha-1}\chi_{[0,t_k]}(s)\leq c(t_m-s)^{(1-\beta)\alpha -1}.
\end{equation*}
\end{lemma}
\begin{proof}\,
We denote the left-hand side by ${\rm I}(s)$.
For any $s\in [t_{i-1},t_i)$, $i\leq m$,
\begin{align*}
  {\rm I}(s) & = \tau \sum_{k=i}^mt_{m-k+1}^{\beta\alpha-1}(t_{k+1}-s)^{(1-2\beta)\alpha-1}\\
              & \leq \tau \sum_{k=i}^mt_{m-k+1}^{\beta\alpha-1}t_{k+1-i}^{(1-2\beta)\alpha-1} \leq ct_{m-i+1}^{(1-\beta)\alpha-1} \leq c(t_m-s)^{(1-\beta)\alpha-1}.
\end{align*}
This completes the proof of the lemma.\qed
\end{proof}

The next result gives a bound on $g(t)=(A(0)-A(t))w(t)$ when $u_0\equiv0$.
\begin{lemma}\label{lem:g-1}
Let $g(t) = (A(0) - A(t))w(t)$ {\rm(}with $u_0\equiv0${\rm)}. Then there holds
\begin{equation}\label{eqn:g-1-1}
  \|g'(0)\|_{L^2(\Omega)} \leq c\|f(0)\|_{L^2(\Omega)},
\end{equation}
and further, for any $\beta\in (0,1/2)$
\begin{align}
   & \tau \sum_{k=1}^mt_{m-k+1}^{\beta\alpha-1}t_k \int_0^{t_k}(t_k-s)^{(1-2\beta)\alpha-1}\|g''(s)\|_{\dot H^{-2\beta}(\Omega)}  \label{eqn:g-1-2}  \\
     \leq &ct_m^{\alpha-1}\|f(0)\|_{L^2(\Omega)}+t_m^{\alpha} \|f'(0)\|_{L^2(\Omega)}+t_m\int_0^{t_m}(t_m-s)^{\alpha-1}\|f''(s)\|_{L^2(\Omega)}\d s.\nonumber
\end{align}
\end{lemma}
\begin{proof}
It follows from Lemma \ref{lem:conti-A} that
\begin{align*}
  \|g'(t)\|_{L^2(\Omega)} & \leq \|(A(0)-A(t))w'(t)\|_{L^2(\Omega)} + \|A'(t)w(t)\|_{L^2(\Omega)}\\
     & \leq ct\|u'(t)\|_{\dot H^2(\Omega)} + c\|u(t)\|_{\dot H^2(\Omega)}.
\end{align*}
Then by  Theorem \ref{thm:reg-inhomo}, $\|g'(0)\|_{L^2(\Omega)} \leq c\|f(0)\|_{L^2(\Omega)}$, showing the
estimate \eqref{eqn:g-1-1}. Next, by Lemma \ref{lem:bdd-kernel}, the left hand side (LHS) of \eqref{eqn:g-1-2} is bounded by
\begin{align*}
  {\rm LHS} & \leq t_m\int_0^{t_m}\Big(\tau\sum_{k=1}^mt_{m-k+1}^{\beta\alpha-1}(t_k-s)^{(1-2\beta)\alpha-1}\chi_{[0,t_k]}(s)\Big)\|g''(s)\|_{H^{-2\beta}(\Omega)}\d s\\
     & \leq ct_m\int_0^{t_m}(t_m-s)^{(1-\beta)\alpha-1}\|g''(s)\|_{H^{-2\beta}(\Omega)}\d s.
\end{align*}
Since $g''(t) = (A(0)-A(t)u''(t) - 2A'(t)u'(t) - A''(t)u(t)$, Theorem \ref{thm:reg-inhomo} implies
\begin{align*}
\|g''(t)\|_{H^{-2\beta}(\Omega)}
&\le ct \|u''(t)\|_{H^{2-2\beta}(\Omega)} + c\| u'(t)\|_{H^{2-2\beta}(\Omega)}+c\| u(t)\|_{H^{2-2\beta}(\Omega)}\\
& \le c t^{\beta\alpha-1} \| f(0) \|_{L^2(\Omega)} +  ct^{\beta\alpha} \| f'(0) \|_{L^2(\Omega)}\\
  &\quad + ct\int_0^t (t-s)^{\beta\alpha-1} \| f''(s) \|_{L^2(\Omega)}\,\d s.
\end{align*}
Combining the last two estimates yields the desired assertion.\qed
\end{proof}

Now we can derive error estimates for the inhomogeneous problem.
\begin{theorem}\label{thm:err-inhomo}
Let $u$ and $u^n$ be the solutions to \eqref{eqn:pde} and \eqref{eqn:fully-correct} with $u_0=0$ and $f\in C^1([0,T];L^2(\Omega))$ and $\int_0^t(t-s)^{\alpha-1}\|f''(s)\|_{L^2(\Omega)}\d s<\infty$, respectively. Then
under conditions \eqref{Cond-1}--\eqref{Cond-2}, there holds
\begin{align*}
  \| u(t_n)-u^n \|_{L^2(\Omega)}   \le& c\tau^2\Big(t_n^{\alpha-2}\|f(0)\|_{L^2(\Omega)}+t_n^{\alpha-1}\|f'(0)\|_{L^2(\Omega)}\\
    &+\int_0^{t_n}(t_n-s)^{\alpha-1}\|f''(s)\|_{L^2(\Omega)}\d s\Big).
\end{align*}
\end{theorem}
\begin{proof}
The overall proof strategy is similar to Theorem \ref{thm:err-smooth}.
First, we bound $\vartheta^n:=\overline w^n-w(t_n)$. By Lemma \ref{lem:sbd-0}(i),
 for any $\beta\in[0,1/2)$, there holds
\begin{align*}
\|\overline{w}_1^n - w_1(t_n) \|_{H^{2\beta}(\Omega)} & \leq  c\tau^2R(t_n).
\end{align*}
with $R(t_n)$ defined by
\begin{align*}
R(t_n) &= t_n^{(1-\beta)\alpha-2}\| f(0)\|_{L^2(\Omega)}  + t_n^{(1-\beta)\alpha-1}\| f'(0)\|_{L^2(\Omega)}\\
 &\quad+ \int_0^{t_n} (t_{n+1}-s)^{(1-\beta)\alpha-1} \|f''(s) \|_{L^2(\Omega)} \,\d s.
\end{align*}
Meanwhile, for any $\beta\in[0,1/2)$, by Lemma \ref{lem:sbd-0}(i) and \eqref{eqn:g-1-1}, with $g(t)=(A(0)-A(t))u(t)$,
\begin{align*}
 &\| \overline{w}_0^n-w_0(t_n) \|_{\dot H^{2\beta}(\Omega)}\\
\le& c\tau^2\Big(t_n^{(1-\beta)\alpha-1}\| f(0)\|_{L^2(\Omega)} + \int_0^{t_n}(t_{n+1}-s)^{(1-2\beta)\alpha-1}\| g''(s)\|_{\dot H^{-2\beta}(\Omega)} \d s\Big).
\end{align*}
Thus, by the splitting \eqref{eqn:split-vartheta} and triangle inequality, for any $\beta\in[0,1/2)$,
\begin{align*}
\| \vartheta^n\|_{\dot H^{2\beta}(\Omega)}\le& c\tau^2R(t_n)
+ c\tau^2\int_0^{t_n}(t_{n+1}-s)^{(1-2\beta)\alpha-1}\|g''(s)\|_{\dot H^{-2\beta}(\Omega)} \d s.
\end{align*}
Next we bound $\varrho^n:=w^n-\overline w^n$, by repeating the argument for Lemma \ref{lem:stability-homo}.
The term ${\rm I}$ can be bounded as \eqref{eqn:bdd-rhoI}. Further, by \eqref{eqn:E-smoothing} and Lemma \ref{lem:conti-A},
for any $\beta\in(0,1/2)$,
\begin{align*}
  {\rm II} &\leq \tau \sum_{k=1}^m\|E_{\tau,m}^{m-k}(A(t_k)-A(0))\vartheta^k\|_{L^2(\Omega)}
  \le c\sum_{k=1}^mt_{m-k+1}^{\beta\alpha-1}t_k\|\vartheta^k\|_{\dot H^{2\beta}(\Omega)}.
\end{align*}
Then the preceding bound on $\vartheta^n$ implies
\begin{align*}
  {\rm II} & \leq c\tau^3 \sum_{k=1}^mt_{m-k+1}^{\beta\alpha-1}t_k\Big(R(t_k) +\int_0^{t_k}(t_{k+1}-s)^{(1-2\beta)\alpha-1}\|g''(s)\|_{\dot H^{-2\beta}(\Omega)} \Big).
\end{align*}
This and \eqref{eqn:g-1-2} imply
\begin{align*}
   \| \varrho^m  \|_{L^2(\Omega)} \le &c \tau\sum_{k=1}^m\|\varrho^k\|_{L^2(\Omega)}+c\tau^2 \Big(t_m^{\alpha-1}\|f(0)\|_{L^2(\Omega)}+t_m^{\alpha}\|f'(0)\|_{L^2(\Omega)}\\
      &+t_m\int_0^{t_n}(t_n-s)^{\alpha-1}\|f''(s)\|_{L^2(\Omega)}\d s\Big).
\end{align*}
Thus, by the discrete Gronwall's inequality from Lemma \ref{lem:gronwall-discrete} with {$\mu=1-\alpha$},
\begin{equation*}
\begin{split}
 \| \varrho^m  \|_{L^2(\Omega)}    \le &c \tau^2\Big(t_m^{\alpha-1}\|f(0)\|_{L^2(\Omega)}+t_m^\alpha \|f'(0)\|_{L^2(\Omega)}\\
   &+t_m\int_0^{t_n}(t_n-s)^{\alpha-1}\|f''(s)\|_{L^2(\Omega)}\d s\Big),
\end{split}
\end{equation*}
{where the constant $c$ depends on $\alpha$ as $O(\alpha^{-1})$.}
This and the bound on $\vartheta^n$ with $\beta=0$ complete the proof.\qed
\end{proof}

\begin{remark}
The error estimate in Theorem \ref{thm:err-inhomo} is identical with that for the
subdiffusion model with a time-independent diffusion coefficient \cite{JinLazarovZhou:SISC2016}.
\end{remark}

\begin{remark}
In the proof of Theorems \ref{thm:err-smooth} and \ref{thm:err-inhomo}, (discrete) Gronwall's inequality
was employed a few times to bound $\varrho^m$. This leads to a dependence on $\alpha$ as $1/\alpha$,
which is nevertheless uniformly bounded on $\alpha$ for $\alpha\to1^-$. Further,
the constants in the bounds on $\vartheta$ are also bounded. Thus, the constants in Theorems \ref{thm:err-smooth}
and \ref{thm:err-inhomo} are bounded as the fractional order $\alpha\to1^-$. 
We refer to \cite{ChenStynes:2019} for an in-depth discussion and many further references on the important issue of $\alpha$-robustness.
\end{remark}

\section{Numerical results and discussions}\label{sec:numer}
Now we present numerical results to illustrate the convergence behavior
of the scheme \eqref{eqn:fully-correct}. To this end, we consider
the domain $\Omega=(0,1)$ and the subdiffusion model \eqref{eqn:pde}
with a time-dependent diffusion operator $A(t)=-(2+
\cos(t))\Delta$. We consider the following three examples:
\begin{itemize}
\item[(a)] $u_0(x)=x^{-1/4}\in H^{1/4-\epsilon}(\Omega)$ with $\epsilon\in(0,1/4)$ and $f\equiv0$.
\item[(b)] $u_0(x)=0$  and $f=e^{t}(1+\chi_{(0,\frac12)}(x))$.
\item[(c)] $u_0(x)=0$  and $f=t^{0.5}x(1-x)$.
\end{itemize}

To discretize the problem, we divide the domain $\Omega$ into $M$ subintervals of equal length $h=1/M$.
The numerical solutions are computed by the standard Galerkin FEM (with P1 element) in space, and BDF2-CQ
in time. Since the spatial convergence was already studied in \cite{JinLiZhou:2017variable},
we only study the temporal convergence below. To this end, we fix a small spatial mesh size $h=1/1000$ so
that the spatial discretization error is negligible, and compute the $L^2(\Omega)$ error:
\begin{equation*}
  e(t_N) =  \|u_h^N-u_h(t_N)\|_{L^2(\Omega)}.
\end{equation*}
Since the exact semidiscrete solution $u_h(t)$ is unavailable, we compute the reference solutions on a finer temporal
mesh with a time stepsize $\tau=1/5000$.

The numerical results for the homogeneous case (a) by the schemes \eqref{eqn:fully} and \eqref{eqn:fully-correct}
are presented in Tables \ref{tab:a} and \ref{tab:a-correct}, respectively. It is clearly observed that
the vanilla BDF2-CQ scheme \eqref{eqn:fully} can only achieve a first-order convergence, whereas the corrected scheme
\eqref{eqn:fully-correct} achieves the desired second-order convergence. The convergence is fairly robust
with respect to the fractional order $\alpha$, despite the low regularity of the initial
data $u_0$. Further, the error is larger when the time $t_N$ gets closer to zero, which agrees well with the regularity theory
in that the second-order temporal derivative of the solution has strong singularity at $t=0$, cf. Theorem \ref{thm:reg-homo}.

\begin{table}[hbt!]
\caption{Temporal errors $e$ for Example (a), uncorrected BDF2-CQ \eqref{eqn:fully} with $\tau =1/N$.}\label{tab:a}
\vspace{0cm}{\setlength{\tabcolsep}{7pt}
	\centering
	\begin{tabular}{ |c|c|cccccc|c|}
		\hline
	$t_N$  & $\alpha\backslash N$ &$10$ &$20$ & $40$ & $80$ & $160$ & $320$  &rate \\
		\hline
	      & $0.25$      & 2.96e-4 & 1.49e-4 & 7.45e-5 & 3.73e-5 & 1.86e-5 & 9.32e-6   & 1.00  \\
     $ 1$ & $0.50$      & 4.12e-4 & 2.12e-4 & 1.07e-4 & 5.37e-5 & 2.69e-5  &1.35e-5   & 1.00  \\
          & $0.75$      & 3.00e-4 & 1.62e-4 & 8.34e-5 & 4.22e-5 & 2.12e-5  &1.06e-5  & 1.00   \\
    \hline
    	 & $0.25$       & 1.16e-3 & 5.80e-4 & 2.89e-4 & 1.45e-4 & 7.23e-5 & 3.62e-5  & 1.00  \\
   $10^{-3}$ & $0.50$     & 5.49e-3 & 2.70e-3 & 1.34e-3 & 6.59e-4 & 3.34e-4 & 1.67e-4   & 1.00  \\
     & $0.75$      & 5.18e-3 & 2.54e-3 & 1.26e-3 & 6.28e-4 & 3.13e-4 & 1.57e-4   & 1.00 \\
    \hline
   \end{tabular}}
\end{table}

\begin{table}[hbt!]
\caption{Temporal errors $e$ for Example (a), corrected BDF2-CQ \eqref{eqn:fully-correct} with $\tau =1/N$.}\label{tab:a-correct}
\vspace{0cm}{\setlength{\tabcolsep}{7pt}
	\centering
	\begin{tabular}{ |c|c|cccccc|c|}
		\hline
	$t_N$  & $\alpha\backslash N$ &$10$ &$20$ & $40$ & $80$ & $160$ & $320$  &rate \\
		\hline
	      & $0.25$      & 4.20e-5 & 9.77e-6 & 2.36e-6 & 5.79e-7 & 1.44e-7 & 3.57e-8   & 2.01  \\
     $ 1$ & $0.50$      & 8.82e-5 & 2.04e-5 & 4.91e-6 & 1.20e-6 & 2.98e-7  &7.41e-8   & 2.01  \\
          & $0.75$      & 1.01e-4 & 2.34e-5 & 5.60e-6 & 1.37e-6 & 3.38e-7  &8.41e-8  & 2.01   \\
    \hline
    	 & $0.25$       & 1.50e-4 & 3.49e-5 & 8.44e-6 & 2.07e-6 & 5.14e-7 & 1.28e-7  & 2.01 \\
   $10^{-3}$ & $0.50$    & 4.77e-4 & 1.13e-4 & 2.74e-5 & 6.77e-6 & 1.68e-6 & 4.19e-7   & 2.00 \\
     & $0.75$      & 3.68e-4 & 8.67e-5 & 2.11e-5 & 5.21e-6 & 1.29e-6 & 3.22e-7   & 2.00  \\
    \hline
   \end{tabular}}
\end{table}

The numerical results for Examples (b) and (c) are presented in Tables \ref{tab:b}--\ref{tab:c-correct},
where the source term $f$ is smooth and nonsmooth in time, respectively. Note that for Example (c), the corrected and
uncorrected schemes are identical, since $f(0)\equiv0$. The observations from Example (a) remain valid
for the inhomogeneous problems: the correction at the first step in the scheme \eqref{eqn:fully-correct}
can restore the desired second-order convergence, whereas the vanilla BDF2--CQ scheme \eqref{eqn:fully}
can only give a first-order convergence, and the convergence rate does not depend on the fractional order $\alpha$.

The second-order convergence of the scheme \eqref{eqn:fully-correct} in Theorem \ref{thm:err-inhomo}
requires suitable temporal regularity of the source $f$, i.e., $\int_0^t(t-s)^{\alpha-1}\|f''(s)\|_{L^2(\Omega)}
\d s<\infty$, in the absence of which, the convergence rate suffers from a loss.
This is clearly observed from the numerical results in Table \ref{tab:c-correct} for Example (c), where
the source term $f$ does not satisfy the condition. Actually, by means of interpolation, the theoretical
convergence rate is $O(\tau^{3/2})$. The corrected scheme \eqref{eqn:fully-correct} can achieve a convergence rate
$O(\tau^{3/2})$, which agrees well with the theoretical one and is faster than the first-order convergence as
exhibited by the scheme \eqref{eqn:fully}. These numerical results show clearly the robustness and
efficiency of the corrected scheme \eqref{eqn:fully-correct}.

\begin{table}[hbt!]
\caption{Temporal errors $e$ for Example (b), uncorrected BDF2-CQ \eqref{eqn:fully} with $\tau =1/N$.}\label{tab:b}
\vspace{0cm}{\setlength{\tabcolsep}{7pt}
	\centering
	\begin{tabular}{ |c|c|cccccc|c|}
		\hline
	$t_N$  & $\alpha\backslash N$ &$10$ &$20$ & $40$ & $80$ & $160$ & $320$  &rate \\
		\hline
	      & $0.25$      & 2.22e-5 & 1.15e-5 & 5.81e-6 & 2.92e-6 & 1.46e-6  & 7.33e-7   & 1.00  \\
     $ 1$ & $0.50$      & 3.09e-5 & 1.64e-5 & 8.35e-6 & 4.21e-6 & 2.11e-6  & 1.06e-6   & 1.00  \\
          & $0.75$      & 2.36e-5 & 1.28e-5 & 6.57e-6 & 3.32e-6 & 1.67e-6  & 8.36e-6  & 1.00   \\
    \hline
      	  & $0.25$       & 9.12e-5 & 4.56e-5 & 2.28e-5 & 1.14e-5 & 5.70e-6 & 2.85e-6  & 1.00  \\
$10^{-3}$ & $0.50$     & 4.32e-4 & 2.12e-4 & 1.05e-4 & 5.26e-5 & 2.62e-5 & 1.31e-5    & 1.00  \\
          & $0.75$      & 3.54e-4 & 1.74e-4 & 8.61e-5 & 4.28e-5 & 2.14e-5 & 1.07e-5   & 1.00 \\
    \hline
   \end{tabular}}
\end{table}

\begin{table}[hbt!]
\caption{Temporal errors $e$ for Example (b), corrected BDF2-CQ \eqref{eqn:fully-correct} with $\tau =1/N$.}\label{tab:b-correct}
\vspace{0cm}{\setlength{\tabcolsep}{7pt}
	\centering
	\begin{tabular}{ |c|c|cccccc|c|}
		\hline
	$t_N$  & $\alpha\backslash N$ &$10$ &$20$ & $40$ & $80$ & $160$ & $320$  &rate \\
		\hline
	      & $0.25$      & 4.46e-6 & 1.04e-6 & 2.51e-7 & 6.16e-8 & 1.53e-8 & 3.80e-9   & 2.01  \\
     $ 1$ & $0.50$      & 8.51e-6 & 1.96e-6 & 4.70e-7 & 1.15e-7 & 2.85e-8 & 7.09e-9   & 2.01  \\
          & $0.75$      & 8.13e-6 & 1.85e-6 & 4.40e-7 & 1.07e-7 & 2.64e-8 & 6.56e-9   & 2.01   \\
    \hline
    	  & $0.25$      & 1.18e-5 & 2.75e-6 & 6.64e-7 & 1.63e-7 & 4.05e-8 & 1.01e-8   & 2.01 \\
$10^{-3}$ & $0.50$      & 3.70e-5 & 8.75e-6 & 2.13e-6 & 5.26e-7 & 1.31e-7 & 3.26e-8   & 2.00 \\
          & $0.75$      & 5.66e-6 & 1.38e-6 & 3.42e-7 & 8.52e-8 & 2.12e-8 & 5.30e-9   & 2.00  \\
    \hline
   \end{tabular}}
\end{table}

\begin{table}[hbt!]
\caption{Temporal errors $e$ for Example (c), corrected BDF2-CQ \eqref{eqn:fully-correct} with $\tau =1/N$.}\label{tab:c-correct}
\vspace{0cm}{\setlength{\tabcolsep}{6pt}
	\centering
	\begin{tabular}{ |c|c|cccccc|c|}
		\hline
	$t_N$  & $\alpha\backslash N$ &$50$ &$100$ & $200$ & $400$ & $800$ & $1600$  &rate \\
		\hline
	      & $0.25$      & 2.31e-8 & 8.65e-9 & 3.18e-9 & 1.15e-9 & 4.11e-10 & 1.44e-10   & 1.52  \\
     $ 1$ & $0.50$      & 2.77e-8 & 1.11e-8 & 4.24e-9 & 1.58e-9 & 5.73e-10 & 2.02e-10   & 1.50   \\
          & $0.75$      & 6.59e-9 & 4.94e-9 & 2.40e-9 & 1.01e-9 & 3.93e-10 & 1.45e-10   & 1.44    \\
    \hline
    	  & $0.25$      & 3.65e-9 & 1.27e-9 & 4.41e-10 & 1.54e-10 & 5.37e-11 & 1.84e-11   & 1.54  \\
$10^{-3}$ & $0.50$      & 1.72e-8 & 5.92e-9 & 2.05e-9 & 7.15e-10 & 2.48e-10 & 8.51e-11   & 1.55  \\
          & $0.75$      & 1.24e-8 & 4.37e-9 & 1.55e-9 & 5.45e-10 & 1.91e-10 & 6.59e-11   & 1.54   \\
    \hline
   \end{tabular}}
\end{table}

\section*{Acknowledgements}
The authors are grateful to two anonymous referees for their constructive comments which have led
to an improvement in the quality of the paper.
\appendix

\section{Gronwall's inequalities}\label{app:gronwall}
In this appendix, we collect several useful Gronwall's inequalities. The following generalized
Gronwall's inequality is useful 
\cite[Exercise 4, p. 190]{Henry:1981}.
\begin{lemma}\label{lem:gronwall}
Let the function $\varphi(t)\geq0$ be continuous for $0< t\leq T$. If
\begin{equation*}
  \varphi(t)\leq a{t^{-\mu}} + b\int_0^t (t-s)^{\beta-1}\varphi(s)\d s,\quad 0<t\leq T,
\end{equation*}
for some constants $a,b\geq0$, $0\leq \mu,\beta<1$, then there is a constant $c=c(b,T,\beta)$ such that
\begin{equation*}
  \varphi(t)\leq {\frac{ac}{1-\mu}t^{-\mu}},\quad 0<t\leq T.
\end{equation*}
\end{lemma}

The next result is a discrete analogue of Lemma \ref{lem:gronwall} \cite[Lemma 7.1]{ElliottLarsson:1992}.
\begin{lemma}\label{lem:gronwall-discrete}
Let $ \varphi^n\geq0$ for $0\leq t_n\leq T$. If
\begin{equation*}
  \varphi^n\leq a{t_n^{-\mu}} + b\tau\sum_{j=1}^{n}\varphi^j,\quad 0<t_n\leq T,
\end{equation*}
for some constants $a,b\geq 0$, and $b\tau<1/2$, $0\leq \mu<1$, then there is constant $c=c(b,T)$ such that
\begin{equation*}
  \varphi^n\leq {\frac{ac}{1-\mu}t_n^{-\mu},}\quad 0<t_n\leq T.
\end{equation*}
\end{lemma}
\begin{proof}
This lemma is a special case of \cite[Lemma 7.1]{ElliottLarsson:1992}, but without explicit
dependence on $\alpha$. We give a short proof for completeness following \cite[p. 258]{Thomee:2006}.
Let $\sigma^n=\tau\sum_{j=1}^n{\color{blue}\varphi^n}$, and $\phi^n=at_n^{-\mu}$. Then
\begin{equation*}
  \tau^{-1}(\sigma^n-\sigma^{n-1}) \leq \phi^n + b\sigma^{n},
\end{equation*}
i.e., $\sigma^n \leq (1-b\tau)^{-1}\sigma^{n-1} + (1-b\tau)^{-1}\tau \phi^n$. Consequently, since the time interval is finite,
\begin{equation*}
  \sigma^n \leq \tau\sum_{j=1}^n (1-b\tau)^{j-n-1}\phi^j \leq e^{2bT}\tau\sum_{j=1}^n\phi^j \leq\frac{ae^{2bT}}{1-\mu}t_n^{1-\mu},
\end{equation*}
since $(1-x)^{-1}\leq e^{2x}$ for $x\in[0,1/2]$.
This directly shows the desired assertion.\qed
\end{proof}

The next result is a variant of Lemma \ref{lem:gronwall-discrete} with log factors \cite[p. 258]{Thomee:2006}.
\begin{lemma}\label{lem:gronwall-discrete-log}
Let $\varphi^n\geq0$ for $0\leq t_n\leq T$. With $\ell_n= \log(1+t_n/\tau)$, if
\begin{equation*}
  \varphi^n\leq at_n^{-1} \ell_n^p + b\tau\sum_{j=1}^n \varphi^j,\quad 0<t_n\leq T,
\end{equation*}
for some constants $a,b\geq 0$ and $p>0$, then there is constant $c=c(b,T)$ such that
\begin{equation*}
  \varphi^n\leq cat_n^{-1}\ell_n^p,\quad 0<t_n\leq T.
\end{equation*}
\end{lemma}

\section{Proof of Theorem \ref{thm:reg-inhomo}}\label{app:inhomo}
In this part, we prove Theorem \ref{thm:reg-inhomo}, by considering $(t^ku(t))^{(k)}$, instead of
$(t^{k+1}u(t))^{(k)}$ for the proof of Theorem \ref{thm:reg-homo}. We begin with a bound on
$\frac{\d ^k}{\d t^k}(t^k\int_0^t E(t-s;t_*) f(s)\d s)$,
which follows from straightforward but lengthy computation.
\begin{lemma}\label{lem:bdd-f} Let $k\geq 1$. Then for any $\beta\in [0,1)$, there holds
\begin{align*}
  & \Big\|A_*^\beta\frac{\d ^k}{\d t^k}\Big(t^k\int_0^t E_*(t-s) f(s)\d s\Big)|_{t=t_*}\Big\|_{L^2(\Omega)}\\
   \leq & c\sum_{m=0}^{k-1} t_*^{(1-\beta)\alpha+m}\|f^{(m)}(0)\|_{L^2(\Omega)} + ct_*^k\int_0^{t_*}(t_*-s)^{(1-\beta)\alpha-1}\|f^{(k)}( s)\|_{L^2(\Omega)}\d s,
\end{align*}
and further,
\begin{align*}
  & \Big\|A_*\frac{\d ^k}{\d t^k}\Big(t^k\int_0^t E_*(t-s) f(s)\d s\Big)|_{t=t_*}\Big\|_{L^2(\Omega)}\\
   \leq & c\sum_{m=0}^k t_*^{m}\|f^{(m)}(0)\|_{L^2(\Omega)} + ct_*^k\int_0^{t_*}\|f^{(k+1)}( s)\|_{L^2(\Omega)}\d s.
\end{align*}
\end{lemma}
\begin{proof}
Let ${\rm I}(t) = \frac{\d ^k}{\d t^k}(t^k\int_0^t E_*(t-s) f(s)\d s)$. It follows from the
elementary identity
\begin{align*}
  \frac{\d^m}{\d t^m} \int_0^t E_*(s)f(t-s)\d s & = \sum_{\ell=0}^{m-1}E_*^{(\ell)}(t)f^{(m-1-\ell)}(0) + \int_0^t E_*(s)f^{(m)}(t-s)\d s
\end{align*}
and direct computation that
\begin{align*}
 {\rm I}(t)  = \sum_{m=0}^k \left(\begin{array}{c}m\\ k\end{array}\right)^2t^m\Big(\sum_{\ell=0}^{m-1}E_*^{(\ell)}(t)f^{(m-1-\ell)}(0) + \int_0^t E_*(s)f^{(m)}(t-s)\d s\Big).
\end{align*}
Consequently, by Lemma \ref{lem:smoothing}, for $\beta\in[0,1)$,
\begin{align*}
   \|A_*^\beta{\rm I}&(t^*)\|_{L^2(\Omega)}
   \leq  c\sum_{m=0}^k t_*^m\sum_{\ell=0}^{m-1} \Big\|A_*^\beta E_*^{(\ell)}(t_*)\Big\|\|f^{(m-1-\ell)}(0)\|_{L^2(\Omega)} \\
    & + c\sum_{m=0}^k t_*^m\int_0^{t_*}\|A_*^\beta E_*(s)\|\|f^{(m)}(t_*-s)\|_{L^2(\Omega)}\d s\\
   \leq& c\sum_{m=0}^k t_*^m\sum_{\ell=0}^{m-1} t_*^{(1-\beta)\alpha-1-\ell}\|f^{(m-1-\ell)}(0)\|_{L^2(\Omega)}
    \\ &+ c\sum_{m=0}^k t_*^m\int_0^{t_*}s^{(1-\beta)\alpha-1}\|f^{(m)}(t_*-s)\|_{L^2(\Omega)}\d s\\
   \leq & c\sum_{m=0}^{k-1}t_*^{(1-\beta)\alpha+m}\|f^{(m)}(0)\|_{L^2(\Omega)}+ c\sum_{m=0}^k t_*^m\int_0^{t_*}s^{(1-\beta)\alpha-1}\|f^{(m)}(t_*-s)\|_{L^2(\Omega)}\d s.
\end{align*}
Next we simplify the second summation. The following elementary identity: for $m<k$,
\begin{equation}\label{eqn:bdd-fs}
  f^{(m)}(s) = \sum_{j=0}^{k-m-1} f^{(m+j)}(0)\frac{s^j}{j!} + \frac{1}{(k-m)!} \int_0^s (s-\xi)^{k-m-1}f^{(k)}(\xi)\d \xi,
\end{equation}
and the semigroup property of the Riemann-Liouville fractional integral operator imply
\begin{align*}
    & \int_0^{t_*}(t_*-s)^{(1-\beta)\alpha-1}\|f^{(m)}(s)\|_{L^2(\Omega)}\d s
   \leq \int_0^{t_*}(t_*-s)^{(1-\beta)\alpha-1}\\
   &\times\Big(\sum_{j=0}^{k-m-1} \|f^{(m+j)}(0)\|_{L^2(\Omega)}\frac{s^j}{j!} + \frac{1}{(k-m)!}\int_0^s (s-\xi)^{k-m-1}\|f^{(k)}(\xi)\|_{L^2(\Omega)}\d \xi\Big)\d s\\
    \leq & c\sum_{j=0}^{k-m-1} t_*^{(1-\beta)\alpha + j}\|f^{(m+j)}(0)\|_{L^2(\Omega)}+ ct_*^{k-m}\int_0^{t_*} (t_*-s)^{(1-\beta)\alpha-1}\|f^{(k)}(s)\|_{L^2(\Omega)}\d s.
\end{align*}
Combining these estimates gives the desired assertion for $\beta\in[0,1)$. For $\beta=1$,
by the identity \eqref{eqn:AE=F'} and integration by parts (and the identity $I-F_*(0)=0$),
\begin{align*}
  A_*\int_0^{t_*} E_*(s)f^{(m)}(t_*-s)\d s
     &= \int_0^{t_*}(I-F_*(s))' f^{(m)}(t_*-s)\d s\\
    & = (I-F_*(t_*))f^{(m)}(0) + \int_0^{t_*} (I-F_*(s))f^{(m+1)}(t_*-s)\d s,
\end{align*}
and thus
\begin{align*}
   \|A_*{\rm I}(t_*)|_{t=t_*}\|_{L^2(\Omega)}
   \leq & c\sum_{m=0}^k t_*^m\Big(\sum_{\ell=0}^{m-1} \Big\|A_* E_*^{(\ell)}(t_*)\Big\|\|f^{(m-1-\ell)}(0)\|_{L^2(\Omega)} \\
     &+ \|f^{(m)}(0)\| + \int_0^{t_*}\|f^{(m+1)}( s)\|_{L^2(\Omega)}\d s\Big).
\end{align*}
Then repeating the preceding argument completes the proof.\qed
\end{proof}

Now we can present the proof of Theorem \ref{thm:reg-inhomo}.
\begin{proof}\,
Similar to Theorem \ref{thm:reg-homo}, the proof is based on mathematical induction.
Let $v_k(t)=t^ku(t)$ and $W_k(t)=t^kE_*(t)$. For $k=0$, by the representation
\eqref{eqn:Sol-expr-u}, we have
\begin{align*}
  A_*^\beta u(t_*) & = \int_0^{t_*}A_*^\beta E_*(t_*-s) f(s)\d s+\int_0^{t_*} A_*^\beta E_*(t_*-s)(A_*-A(s))u(s)\d s .
\end{align*}
Then for $\beta\in[0,1)$, by Lemma \ref{lem:smoothing}(ii) and \ref{lem:conti-A} there holds
\begin{align*}
  \|A_*^\beta u(t_*)\|_{L^2(\Omega)} & \leq  \int_0^{t_*}\|A_*^\beta E_*(t_*-s)\|\|f(s)\|_{L^2(\Omega)}\d s \\
   &\quad +\int_0^{t_*} \|A_* E_*(t_*-s)\|\|A_*^{\beta}(I-A_*^{-1}A(s))u(s)\|_{L^2(\Omega)}\d s\\
    & \leq c\int_0^{t_*}(t_*-s)^{(1-\beta)\alpha-1}\|f(s)\|_{L^2(\Omega)}\d s+c\int_0^{t_*} \|A_*^\beta u(s)\|_{L^2(\Omega)}\d s .
\end{align*}
The case $\beta=1$ follows similarly from the identity \eqref{eqn:AE=F'} and integration by parts:
\begin{align*}
  \|A_*u(t_*)\|_{L^2(\Omega)}  \leq c\|f(0)\|_{L^2(\Omega)} + c\int_0^{t_*}\|f'(s)\|_{L^2(\Omega)}\d s + c\int_0^{t_*} \|A_*u(s)\|_{L^2(\Omega)}\d s.
\end{align*}
Then standard Gronwall's inequality gives the assertion for the case $k=0$.
Now suppose the assertion holds up to $k-1<K$, and we prove it for $k\leq K$. Now note that
\begin{align*}
  v_k^{(k)}(t) &=  \frac{\d ^k}{\d t^k}\Big(t^k\int_0^t E_*(t-s) f(s)\d s\Big) + \sum_{m=0}^k\left(\begin{array}{c} k\\ m\end{array}\right)\frac{\d^k}{\d t^k}\int_0^tW_m(t-s)(A_*-A(s))v_{k-m}(s)\d s.
\end{align*}
This, Lemmas \ref{lem:bdd-f} and \ref{lem:inhomo} and triangle inequality give
\begin{align*}
    &\|A_*^\beta v_k^{(k)}(t)|_{t=t_*}\|_{L^2(\Omega)} \leq  c\int_0^{t_*}\|A_*^\beta v_k^{(k)}(s)\|_{L^2(\Omega)}\d s\nonumber\\
     &\quad + \left\{\begin{aligned}
      c\sum_{m=0}^{k-1} t_*^{(1-\beta)\alpha+m}\|f^{(m)}(0)\|_{L^2(\Omega)} + ct_*^k\int_0^{t_*}(t_*-s)^{(1-\beta)\alpha-1}\|f^{(k)}( s)\|_{L^2(\Omega)}\d s, & \quad \beta\in[0,1),\\
      c\sum_{m=0}^k t_*^{m}\|f^{(m)}(0)\|_{L^2(\Omega)} + ct_*^k\int_0^{t_*}\|f^{(k+1)}( s)\|_{L^2(\Omega)}\d s, &\quad \beta=1.
   \end{aligned}\right.
\end{align*}
This and the standard Gronwall's inequality complete the induction step.\qed
\end{proof}

The following result is needed in the proof of Theorem \ref{thm:reg-inhomo}.
\begin{lemma}\label{lem:inhomo}
Under the conditions in Theorem \ref{thm:reg-inhomo}, for any $\beta\in[0,1]$
and $m=0,\ldots,k$, there holds
\begin{align*}
&\Big\|A_*^\beta \frac{\d^k}{\d t^k}\int_0^t(t-s)^{k-m} E_*(t-s) (A_*-A(s))s^{m}u(s) \d s|_{t=t_*}\Big\|_{L^2(\Omega)} \\
& \le  c_0\int_0^{t_*}\|A_*^\beta(s^ku(s))^{(k)}\|_{L^2(\Omega)}\d s\nonumber\\
     &\quad + \left\{\begin{aligned}
      c\sum_{m=0}^{k-1} t_*^{(1-\beta)\alpha+m}\|f^{(m)}(0)\|_{L^2(\Omega)} + ct_*^k\int_0^{t_*}(t_*-s)^{(1-\beta)\alpha-1}\|f^{(k)}( s)\|_{L^2(\Omega)}\d s, & \quad \beta\in[0,1),\\
      c\sum_{m=0}^k t_*^{m}\|f^{(m)}(0)\|_{L^2(\Omega)} + ct_*^k\int_0^{t_*}\|f^{(k+1)}( s)\|_{L^2(\Omega)}\d s, &\quad \beta=1.
   \end{aligned}\right.
\end{align*}
\end{lemma}
\begin{proof}
Let $v_k=t^ku(t)$ and $W_k(t)=t^kE_*(t)$. By the induction hypothesis and \eqref{eqn:bdd-fs}, for $\ell<m$, we have
\begin{align}\label{eqn:dudt}
  \|A_*v_m^{(\ell)}(s)\|_{L^2(\Omega)}
  &\leq  cs^{m-\ell}\Big(\sum_{j=0}^{\ell}s^{j}\|f^{(j)}(0)\|_{L^2(\Omega)}+s^\ell\int_0^s\|f^{(\ell+1)}(\xi)\|_{L^2(\Omega)}\d \xi\Big)\nonumber\\
  &\leq cs^{m-\ell}\Big(\sum_{j=0}^{k-1}s^{j}\|f^{(j)}(0)\|_{L^2(\Omega)}+s^{k-1}\int_0^s\|f^{(k)}(\xi)\|_{L^2(\Omega)}\d \xi\Big).
\end{align}
We denote the term in the bracket by ${\rm T}(s;f,k)$. Now similar to the proof of Lemma \ref{lem:homo}, let ${\rm I}_m(t)$
be the integral on the left hand side. Then in view of the identity
\begin{equation*}
  {\rm I}_m^{(k)}(t) = \sum_{\ell=0}^{k-m}\Big(\begin{array}{c}\ell\\ k-m\end{array}\Big)\underbrace{\int_0^tW_m^{(m)}(A_*-A(t_*-s))^{(k-m-\ell)}v_{k-m}^{(\ell)}(s)\d s}_{{\rm I}_{m,\ell}(t)},
\end{equation*}
it suffices to bound the integrand $\widetilde{\rm I}_{m,\ell}(s)$
of the integral ${\rm I}_{m,\ell}(t_*)$, $\ell=0,1,\ldots,k-m$. Below we discuss the cases $\beta\in [0,1)$ and
$\beta=1$ separately, due to the difference in singularity, as in the proof of Lemma \ref{lem:homo}.

\noindent{\it Case (i): $\beta\in[0,1)$.}  For the case $\ell<k$, Lemmas \ref{lem:smoothing}(ii) and \ref{lem:conti-A} lead to
\begin{align*}
   \|A_*^\beta \widetilde {\rm I}_{m,\ell}(s)\|_{L^2(\Omega)}
  &\leq\|A_*^\beta W_{m}^{(m)}\|
  \|(A_*-A(t_*-s))^{(k-m-\ell)}v_{k-m}^{(\ell)}(t_*-s)\|_{L^2(\Omega)}\\
  & \leq \left\{\begin{aligned}
     cs^{(1-\beta)\alpha-1}s\|A_*v_{k-m}^{(k-m)}(t_*-s)\|_{L^2(\Omega)}, & \quad\ell=k-m,\\
     cs^{(1-\beta)\alpha-1}\|A_*v_{k-m}^{(\ell)}(t_*-s)\|_{L^2(\Omega)}, &\quad \ell <k-m,
     \end{aligned}\right.\\
     & \leq \left\{\begin{aligned} cs^{(1-\beta)\alpha} {\rm T}(t_*-s;f,k), & \quad \ell = k-m, \\
      cs^{(1-\beta)\alpha-1} (t_*-s)^{k-m-\ell}{\rm T}(t_*-s;f,k), &\quad \ell <k-m,
      \end{aligned}\right.
\end{align*}
where the last step is due to \eqref{eqn:dudt}. Note that for $\ell<k$,
the derivation requires $\beta\in[0,1)$. Similarly, for the case $\ell=k$ (and thus $m=0$),
\begin{align}\label{eqn:bdd-ikk}
 \|A_*^\beta \widetilde {\rm I}_{0,k}(s)\|_{L^2(\Omega)}\leq
    c\|A_*^\beta v_k^{(k)}(t_*-s)\|_{L^2(\Omega)}.
\end{align}

\noindent{\it Case (ii): $\beta=1$.} 
Note that for $\ell<k$, the derivation in case (i) requires $\beta\in[0,1)$. When $\ell<k$
and $\beta=1$, using the identity \eqref{eqn:int-part} and Lemma \ref{lem:smoothing} and repeating
the argument of Lemma \ref{lem:homo}, we obtain
\begin{align*}
   \|A_*\widetilde{\rm I}_{m,\ell}(s)\|_{L^2(\Omega)} & \leq \left\{\begin{aligned} c(t_*-s)^{k-m-\ell-1}{\rm T}(t_*-s;f,k), & \quad \ell < k-1,\\
    c{\rm T}(t_*-s;f,k) + c\|A_*v_k^{(k)}(t_*-s)\|_{L^2(\Omega)}, &\quad \ell = k-1,
    \end{aligned}\right.
\end{align*}
Combining the preceding estimates, integrating from $0$ to $t_*$ in $s$ and then applying
standard Gronwall's inequality complete the proof.\qed
\end{proof}

\bibliographystyle{abbrv}

\end{document}